\documentclass[11pt]{article}
\usepackage{epigamath}

\usepackage[notext]{kpfonts}
\usepackage{baskervald}

\setpapertype{A4}

\usepackage[english]{babel}

\usepackage[dvips]{graphicx}     

\title{Equivariant perverse sheaves on Coxeter \\ arrangements and buildings}
\titlemark{Sheaves on Coxeter arrangements and buildings}
\author{Martin H. Weissman}
\authoraddresses{
\authordata{Martin H. Weissman}{\firstname{Martin H.} \lastname{Weissman}\\
\institution{Department of Mathematics, University of California, Santa Cruz, CA 95064, USA}\\
\email{weissman@ucsc.edu}}
}
\authormark{M. H. Weissman}
\date{\vspace{-5ex}} 
\journal{\'Epijournal de G\'eom\'etrie Alg\'ebrique} 
\acceptation{Received by the Editors on March 7, 2018, and in final form
on January 22, 2019. \\ Accepted on May 28, 2019.}

\newcommand{\articlereference}{Volume 3 (2019), Article Nr. 8}

\usepackage[all]{xy}

\allowdisplaybreaks

\usepackage{tikz}
\usepackage{tikz-cd}

\makeatletter
\renewcommand\thesubsection{{\thesection}\@arabic\c@subsection.}
\renewcommand{\p@subsection}{\arabic{section}.\arabic{subsection}\expandafter\@gobble}
\makeatother

\numberwithin{equation}{subsection}

\makeatletter

\renewcommand{\p@subsection}{\arabic{section}.\arabic{subsection}.\arabic{equation}\expandafter\@gobble}
\makeatother

\newtheorem{thm}{Theorem}[subsection]

\newtheorem{lemma}[thm]{Lemma}
\newtheorem{proposition}[thm]{Proposition}

\newtheorem{tdefinition}[thm]{Definition}
\newenvironment{definition}{\begin{tdefinition} \rm}{\end{tdefinition}}

\newtheorem{tremark}[thm]{Remark}
\newenvironment{remark}{\begin{tremark} \rm}{\end{tremark}}

\usepackage{enumitem}
\setlist[enumerate,1]{label={\rm (\arabic*)}}

\DeclareMathAlphabet{\mathcalligra}{T1}{calligra}{m}{n}
\DeclareMathOperator{\Vol}{vol}

\usepackage[mathscr]{euscript}
\usepackage[T2A,T1]{fontenc}

  {\begin{description}%
    \setlength{\itemsep}{2.5pt}%
    \setlength{\parskip}{5pt}}%
  {\end{description}}

\DeclareMathOperator{\fd}{{fd}}

\DeclareMathOperator{\Star}{{Star}}

\DeclareMathOperator{\mon}{mon}
\DeclareMathOperator{\trans}{tran}

\DeclareMathOperator{\sm}{sm}

\DeclareMathOperator{\Aut}{Aut}

\DeclareMathOperator{\End}{End}

\DeclareMathOperator{\Ker}{Ker}

\DeclareMathOperator{\Id}{Id}
\DeclareMathOperator{\Res}{Res}
\DeclareMathOperator{\Inc}{Inc}

\newcommand{\From}{\colon}
\newcommand{\inar}{\ar@{^{(}->}}
\newcommand{\onar}{\ar@{->>}}

\newcommand{\defined}[1]{\underline{{#1}}}
\usepackage{accents}
\newlength{\dtildeheight}

\newcommand{\Facet}{\mathcal{F}}

\newcommand{\Hyper}{\mathcal{H}}
\newcommand{\Apart}{\mathcal{A}}
\newcommand{\Build}{\mathcal{B}}

\newcommand{\mm}{\mathfrak{m}}

\newcommand{\Matrix}[4]{ \left( \begin{array}{cc}  #1 & #2 \\  #3 & #4 \\ \end{array} \right) }

\makeatletter
\newcommand{\raisemath}[1]{\mathpalette{\raisem@th{#1}}}
\newcommand{\raisem@th}[3]{\raisebox{#1}{$#2#3$}}
\makeatother

\newcommand{\Cat}[1]{ {\mathsf{#1}} }

\newcommand{\Lie}[1]{ {\mathfrak{#1}} }

\newcommand{\sch}[1]{\underline{\boldsymbol{ \mathrm{#1}}}}
\newcommand{\alg}[1]{\boldsymbol{\mathrm{#1}}}

\newcommand{\sheaf}[1]{{\mathscr{#1}}}

\newcommand{\RR}{\mathbb R}

\newcommand{\CC}{\mathbb C}

\renewcommand{\AA}{\mathcal A}

\newcommand{\FF}{\mathbb F}
\newcommand{\OO}{\mathcal{O}}

\newcommand{\Into}{\hookrightarrow}
\newcommand{\Onto}{\twoheadrightarrow}
\newcommand{\To}{\rightarrow}

\newcommand{\e}{\textbf{e}}

\newcommand{\isom}{\cong}

\makeatletter
\newcommand\@biprod[1]{%
  \vcenter{\hbox{\ooalign{$#1\prod$\cr$#1\coprod$\cr}}}}
\newcommand\biprod{\mathop{\mathpalette\@biprod\relax}\displaylimits}
\makeatother

\newcommand{\defeq}{:=}
\DeclareMathAlphabet{\mathcalligra}{T1}{calligra}{m}{n}

\begin{document}


\maketitle



\begin{prelims}

\vspace{-0.55cm}

\def\abstractname{Abstract}
\abstract{When $W$ is a finite Coxeter group acting by its reflection representation on $E$, we describe the category $\Cat{Perv}_W(E_\CC, \Hyper_\CC)$ of $W$-equivariant perverse sheaves on $E_\CC$, smooth with respect to the stratification by reflection hyperplanes.  By using Kapranov and Schechtman's recent analysis of perverse sheaves on hyperplane arrangements, we find an equivalence of categories from $\Cat{Perv}_W(E_\CC, \Hyper_\CC)$ to a category of finite-dimensional modules over an algebra given by explicit generators and relations.

We also define categories of equivariant perverse sheaves on affine buildings, e.g., $G$-equivariant perverse sheaves on the Bruhat--Tits building of a $p$-adic group $G$.  In this setting, we find that a construction of Schneider and Stuhler gives equivariant perverse sheaves associated to depth zero representations.}




\languagesection{Fran\c{c}ais}{%

\vspace{-0.05cm}
\textbf{Titre. Faisceaux pervers \'equivariants sur les immeubles et arrangements de Coxeter} \commentskip \textbf{R\'esum\'e.} Lorsque $W$ est un groupe de Coxeter fini agissant par r\'eflexions sur $E$, nous d\'ecrivons la cat\'egorie $\Cat{Perv}_W(E_\CC, \Hyper_\CC)$ des faisceaux pervers $W$-\'equivariants sur $E_\CC$, lisses par rapport \`a la stratification par les hyperplans de r\'eflexion. En utilisant l'analyse r\'ecente de Kapranov et Schechtman des faisceaux pervers sur les arrangements d'hyperplans, nous donnons une \'equivalence de cat\'egorie entre $\Cat{Perv}_W(E_\CC, \Hyper_\CC)$ et une cat\'egorie de modules de dimension finie sur une alg\`ebre donn\'ee explicitement par g\'en\'erateurs et relations. 

Nous d\'efinissons \'egalement des cat\'egories de faisceaux pervers \'equivariants sur les immeubles affines, par exemple celle des faisceaux pervers $G$-\'equivariants sur l'immeuble de Bruhat--Tits d'un groupe $p$-adique $G$. Dans ce cadre, nous montrons qu'une construction de Schneider et Stuhler fournit des faisceaux pervers \'equivariants associ\'es \`a des repr\'esentations de profondeur~nulle.}

\end{prelims}


\newpage

\setcounter{tocdepth}{1} \tableofcontents

\section*{Introduction}

In their recent paper \cite{KapSch}, Kapranov and Schechtman give a new description of the category of perverse sheaves on a complex affine space, smooth with respect to a hyperplane arrangement, when the hyperplane arrangement arises as the complexification of a real hyperplane arrangement.  Their description is explicit, in terms of the facets arising from the real hyperplane arrangement.  

An important example is given by the arrangement of reflection hyperplanes, for a finite (or affine) Coxeter group $W$.  In this setting, we have not only a hyperplane arrangement, but a $W$-equivariant hyperplane arrangement; thus we may study the category of $W$-equivariant perverse sheaves.  Theorem \ref{mainth1} of this paper applies the results of \cite{KapSch} to give an explicit description of this category, as a category of modules over a finitely-presented algebra.  This is closely related to remarks of \cite[\S 4.B]{KapSch2}, but while they discuss perverse sheaves on a geometric quotient $\Lie{h}/W$, we consider $W$-equivariant perverse sheaves on $\Lie{h}$; effectively we work on the stacky quotient $[\Lie{h} / W]$, which carries a bit more information.  Moreover, in loc.~cit., Kapranov and Schechtman state that ``A complete quiver description of $\mathrm{Perv}(\Lie{h} / W)$ is not yet available.''  So in this way, the current paper addresses a gap identified by Kapranov and Schechtman.

Our pursuit of this topic began as an effort to revisit the work of Schneider and Stuhler \cite{SS}, who study equivariant sheaves and cosheaves (coefficient systems) on the Bruhat--Tits building of a $p$-adic group.  In their introduction, they write
\begin{quote}
the latter objects [certain coefficient systems] constitute something which one might call perverse sheaves on the building $X$.  From this point of view our constructions bear a certain resemblance to the Beilinson-Bernstein localization theory from Lie algebra representations to perverse sheaves on the flag manifold.''
\end{quote}
In this way, the vision of Schneider and Stuhler is to study something like perverse sheaves on the Bruhat--Tits building.

After Kapranov and Schechtman \cite{KapSch}, there seems to be a natural way to advance Schneider and Stuhler's vision.  For the Bruhat--Tits building is a union of (affine) hyperplane arrangements, glued along facets.  While the Bruhat--Tits building itself does not seem to admit a complexification, a combinatorial notion of perverse sheaf from Kapranov and Schechtman (a monotonic, transitive, invertible bisheaf) generalizes easily enough from a single hyperplane arrangement to a complex thereof (such as an affine building).

We conclude this paper with Theorem \ref{depthzeroPS}, giving an exact functor from the category of depth-zero smooth representations of a $p$-adic group to the category of $G$-equivariant perverse sheaves on the Bruhat--Tits building.  We hope this is the first step in a program of studying $G$-equivariant perverse sheaves on the building.  Constructing a faithful functor from higher-depth representations to perverse sheaves seems difficult, and closely related to theories of minimal K-types and Hecke algebras for $p$-adic groups.

\subsection*{Acknowledgments}

I  thank Jessica Fintzen for a close reading, valuable edits, and crucial insights on Moy-Prasad filtrations -- our discussions were most helpful in the later sections of this paper.  Nicholas Proudfoot provided advice on perverse sheaves and hyperplane arrangements.  I appreciate the expert feedback of Geordie Williamson and Mikhail Kapranov, who informed me about the current state of the art.  I thank Asilata Bapat for pointing out an essential error in the main theorem, in my reformulation of invertibility -- her observation and explicit counterexample allowed me to fix this error.

This paper was completed at the Weizmann Institute of Science, and I appreciate their hospitality during a summer visit.  Collaborations related to this paper were supported by a grant from the Simons Foundation (\#426453).

\section{Bisheaves}

We begin with an abstract treatment of bisheaves and equivariant bisheaves on posets.  The results of this section are somewhat tedious exercises in diagram chasing, but seem necessary for what comes later.  

\subsection{Bisheaves on a poset}

Let $(\Facet, \leq)$ be a partially ordered set (poset).  Elements of $\Facet$ will be called facets, to be consistent with what comes later.  Let $\Omega$ be a category -- we call it the \defined{category of coefficients}.  A common category of coefficients is the category $\Cat{Vec}_\CC^{\fd}$ of finite-dimensional complex vector spaces and $\CC$-linear maps.

\begin{definition}
A $\Omega$-valued \defined{bisheaf} $(\sheaf{V}, \gamma, \delta)$ on $\Facet$ consists of the following data:
\begin{itemize}
\item
For each facet $F \in \Facet$, an object $\sheaf{V}_F \in \Omega$;
\item
For each $F_1 \leq F_2$, a morphism (called an \defined{upper}) $\gamma_{F_1, F_2} \From \sheaf{V}_{F_1} \To \sheaf{V}_{F_2}$;
\item
For each $F_1 \leq F_2$, a morphism (called a \defined{downer}) $\delta_{F_2, F_1} \From \sheaf{V}_{F_2} \To \sheaf{V}_{F_1}$.
\end{itemize}
To be a bisheaf, we require the following axioms:  if $F_1 \leq F_2 \leq F_3$, then 
$$\gamma_{F_2, F_3} \circ \gamma_{F_1, F_2} = \gamma_{F_1, F_3} \text{ and } \delta_{F_2, F_1} \circ \delta_{F_3, F_2} = \delta_{F_3, F_1}.$$
Furthermore, $\delta_{F,F} = \gamma_{F,F} = \Id_{\sheaf{V}_F}$ for all $F \in \Facet$.
\end{definition}

Define $\Cat{BiSh}(\Facet, \Omega)$ to be the category whose objects are bisheaves, and whose morphisms are those families of morphisms that intertwine uppers with uppers and downers with downers accordingly.  

Bisheaves on $\Facet$ can be viewed as functors in the following way:  begin with the directed graph with vertex-set $\Facet$ and two arrows $F_1 \xrightarrow{\gamma} F_2$, $F_2 \xrightarrow{\delta} F_1$ for every relation of the form $F_1 \leq F_2$.  When $F_1 \leq F_2$, we call the arrow $F_1 \xrightarrow{\gamma} F_2$ an \defined{up-arrow}, and $F_2 \xrightarrow{\delta} F_1$ a \defined{down-arrow} (so there is both an up-arrow and a down-arrow from every facet $F$ to itself).

Let $\Cat{Fac}$ be the category with object set $\Facet$, with morphisms generated by this directed graph, modulo the relations:
\begin{enumerate}
\item[(A)]
$F \xrightarrow{\gamma} F = F \xrightarrow{\delta} F = \Id_F$ for all $F \in \Facet$.
\item[(B)]
$(F_2 \xrightarrow{\gamma} F_3) \circ (F_1 \xrightarrow{\gamma} F_2) = (F_1 \xrightarrow{\gamma} F_3)$, for all $F_1 \leq F_2 \leq F_3 \in \Facet$.
\item[(C)]
$(F_2 \xrightarrow{\delta} F_1) \circ (F_3 \xrightarrow{\delta} F_2) = (F_3 \xrightarrow{\delta} F_1)$, for all $F_1 \leq F_2 \leq F_3 \in \Facet$.
\end{enumerate}
We refer to MacLane's text \cite[Chapter II.8]{Mac} for more detail on the construction of categories by generators and relations.  

By construction, the category $\Cat{BiSh}(\Facet, \Omega)$ of $\Omega$-valued bisheaves on $\Facet$ is equivalent (isomorphic, in fact) to the category of functors $\Cat{Fun}(\Cat{Fac}, \Omega)$ from $\Cat{Fac}$ to $\Omega$.  Thus we view bisheaves as functors when convenient.  From this perspective, it follows that if $\Omega$ happens to be an abelian category, then $\Cat{BiSh}(\Facet, \Omega)$ inherits the structure of an abelian category as well.

\subsection{Equivariant structures}

Now we consider equivariant bisheaves.  Suppose that a group $G$ acts on the poset $(\Facet, \leq)$.  Thus if $F_1 \leq F_2$ and $g \in G$, then $g F_1 \leq g F_2$.

If $(\sheaf{V}, \gamma, \delta)$ is a $\Omega$-valued bisheaf on $(\Facet, \leq)$, then a \defined{$G$-equivariant structure} consists of a family $\eta$ of isomorphisms
$$\eta_{g,F} \From \sheaf{V}_F \To \sheaf{V}_{gF} \text{ for all } g \in G, F \in \Facet,$$
such that
\begin{enumerate}
\item[(Eq1)]
$\eta_{1,F} = \Id_{\sheaf{V}_F}$ for all $F \in \Facet$;
\item[(Eq2)]
For all $g,h \in G$ and all $F \in \Facet$, $\eta_{g, hF} \circ \eta_{h,F} = \eta_{gh,F} \From \sheaf{V}_F \To \sheaf{V}_{gh F}$;
\item[(Eq3)]
For all $F_1 \leq F_2$, and all $g \in G$,
$$\eta_{g,F_2} \circ \gamma_{F_1, F_2} = \gamma_{g F_1, g F_2} \circ \eta_{g,F_1};$$
$$\eta_{g,F_1} \circ \delta_{F_2, F_1} = \delta_{g F_2, gF_1} \circ \eta_{g,F_2}.$$
\end{enumerate}

Let $\Cat{BiSh}_G(\Facet, \Omega)$ be the resulting category of \defined{$G$-equivariant $\Omega$-valued bisheaves} on $(\Facet, \leq)$ and $G$-intertwining morphisms of bisheaves.  

The category of $G$-equivariant bisheaves can also be realized as a functor category.  Indeed, consider the directed graph with vertex set $\Facet$, and three types of arrows:  the up-arrows $\xrightarrow{\gamma}$ and down-arrows $\xrightarrow{\delta}$ from before, and a new set of arrows called ``$G$-arrows'' $F \xrightarrow{g} gF$ for every $F \in \Facet$ and every $g \in G$.  Let $\Cat{Fac}_G$ be the category with object set $\Facet$, with morphisms generated by this directed graph, modulo the previous relations (A-C) for up- and down-arrows, together with:
\begin{enumerate}
\item[(D)] $(F \xrightarrow{1} F) = \Id_F$ for all $F \in \Facet$;
\item[(E)] $(h F \xrightarrow{g} g h F) \circ (F \xrightarrow{h} hF) = (F \xrightarrow{gh} gh F)$ for all $g,h \in G$, $F \in \Facet$;
\item[(F)] $(F_2 \xrightarrow{g} g F_2)  \circ (F_1 \xrightarrow{\gamma} F_2) = (g F_1 \xrightarrow{\gamma} g F_2) \circ (F_1 \xrightarrow{g} g F_1)$ for all $g \in G$, $F_1 \leq F_2 \in \Facet$;
\item[(G)] $(F_1 \xrightarrow{g} g F_1) \circ (F_2 \xrightarrow{\delta} F_1) = (g F_2 \xrightarrow{\delta} g F_1) \circ (F_2 \xrightarrow{g} g F_2)$ for all $g \in G$, $F_1 \leq F_2 \in \Facet$.
\end{enumerate}

By construction, the category $\Cat{BiSh}_G(\Facet, \Omega)$ is isomorphic to the the category of functors $\Cat{Fun}(\Cat{Fac}_G, \Omega)$.  Thus, if $\Omega$ is an abelian category, then $\Cat{BiSh}_G(\Facet, \Omega)$ inherits the structure of an abelian category.

\subsection{Restriction of equivariant bisheaves}

Suppose that $G$ acts on $(\Facet, \leq)$ as before.  Let $\Facet^+$ be a subposet of $\Facet$, i.e., a subset with the inherited partial order.  
\begin{definition}
A $\Omega$-valued $G$-equivariant bisheaf on $\Facet^+$ is a bisheaf $(\sheaf{V}, \gamma, \delta)$ on the partially ordered set $\Facet^+$, together with a family of maps $\eta_{g, F} \From \sheaf{V}_F \To \sheaf{V}_{gF}$ indexed by pairs $g \in G, F \in \Facet^+$ for which $gF \in \Facet^+$, satisfying the equivariance axioms (Eq1), (Eq2), (Eq3) in {\em all cases where the maps are defined (within $\Facet^+$)}.

Write $\Cat{BiSh}_G(\Facet^+, \Omega)$ for the resulting category of $G$-equivariant bisheaves on $\Facet^+$.  
\end{definition}

If $(\sheaf{V}, \gamma, \delta, \eta)$ is a $G$-equivariant bisheaf on $\Facet$, and $\Facet^+$ is a subposet, restriction yields a $G$-equivariant bisheaf on $\Facet^+$.  This gives a functor,
$$\Res \From \Cat{BiSh}_G(\Facet, \Omega) \To \Cat{BiSh}_G(\Facet^+, \Omega).$$


One may give a functorial interpretation of $G$-equivariant bisheaves on $\Facet^+$, in the setting above.  For this, consider the graph whose vertex set is $\Facet^+$, with up-arrows $F \xrightarrow{\gamma} F'$ and down-arrows $F' \xrightarrow{\delta} F$ for all $F \leq F'$ (when $F,F' \in \Facet^+$), and with ``$G$-arrows'' $F \xrightarrow{\eta_g} gF$ for all $g \in G$, $F \in \Facet^+$ (when $F, gF \in \Facet^+$).  Let $\Cat{Fac}_G^+$ to be the category whose object set is $\Facet^+$, generated by this directed graph modulo the relations (A) -- (G) whenever the facets are elements of $\Facet^+$.

By construction, there is an isomorphism of categories,
$$\Cat{BiSh}_G(\Facet^+, \Omega) \xrightarrow{\sim} \Cat{Fun}(\Cat{Fac}_G^+, \Omega).$$

There is an obvious functor from $\Cat{Fac}_G^+ \To \Cat{Fac}_G$.  The objects of $\Cat{Fac}_G^+$ form a subset of the objects of $\Cat{Fac}_G$.  The generating set of morphisms of $\Cat{Fac}_G^+$ is a subset of the generating set of morphisms of $\Cat{Fac}_G^+$, and the relations are a subset of the relations.  But since the relations of $\Cat{Fac}_G^+$ are, a priori, a subset of those arising from $\Cat{Fac}_G$, it is not clear whether $\Cat{Fac}_G^+$ is a subcategory of $\Cat{Fac}_G$.  I.e., faithfulness is not a consequence of the construction.

But through the functor $\Cat{Fac}_G^+ \To \Cat{Fac}_G$, restriction of $G$-equivariant bisheaves from $\Facet$ to $\Facet^+$ corresponds to pullback of functors,
$$\Res \From \Cat{Fun}(\Cat{Fac}_G, \Omega) \To \Cat{Fun}(\Cat{Fac}_G^+, \Omega).$$

In general, this restriction is not very well-behaved.  But when $\Facet^+$ carries enough information from $\Facet$, restriction does not lose much information.  
\begin{definition}
A subset $\Facet^+ \subset \Facet$ is called a \defined{$G$-full subposet} if it satisfies the following conditions.
\begin{enumerate}
\item
(closedness):  if $F,F' \in \Facet$, $F \leq F'$, and $F' \in \Facet^+$, then $F \in \Facet^+$.
\item
($G$-fullness)  If $F \in \Facet$, then there exists $g \in G$ such that $g F \in \Facet^+$.
\end{enumerate}
\end{definition}

\begin{thm}
\label{restful}
Suppose that $G$ acts on $(\Facet, \leq)$, and $\Facet^+$ is a $G$-full subposet of $\Facet$.  Then restriction gives an equivalence of categories,
$$\Cat{BiSh}_G(\Facet, \Omega) \xrightarrow{\sim} \Cat{BiSh}_G(\Facet^+, \Omega).$$
\end{thm}
\begin{proof}
It is easier to prove this theorem by viewing both categories as functor categories.  From this perspective, restriction is given by pullback of functors,
$$\Res \From \Cat{Fun}(\Cat{Fac}_G, \Omega) \To \Cat{Fun}(\Cat{Fac}_G^+, \Omega).$$
To show that $\Res$ is an equivalence, we demonstrate that the inclusion functor $\Inc \From \Cat{Fac}_G^+ \To \Cat{Fac}_G$ is an equivalence.  And for this, we construct an inverse functor,
$$S \From \Cat{Fac}_G \To \Cat{Fac}_G^+.$$

The construction depends on the choice, for every $F \in \Facet$, of an element $n_F \in G$ such that $n_F F \in \Facet^+$.  Such an element $n_F$ is not uniquely determined by $F$, but we write $F^+ = n_F F$ in what follows.  When $F \in \Facet^+$ already, we choose $n_F = 1$.  

The functor $S$ is defined on objects by $S(F) = F^+$.  For morphisms, we must be more careful.
\begin{enumerate}
\item
If $F_1 \leq F_2$, write $n_1 = n_{F_1}$ and $n_2 = n_{F_2}$ for the moment.  Then $n_2 F_1 \leq n_2 F_2$, and thus $n_2 F_1 \in \Facet^+$ (by closedness).  Therefore $F_1^+ \xrightarrow{n_2 n_1^{-1}} n_2 F_1$ is a $G$-morphism in $\Cat{Fac}_G^+$.  Define the functor $S$ on up- and down-arrows by
$$S( F_1 \xrightarrow{\gamma} F_2 ) = (n_2 F_1 \xrightarrow{\gamma} F_2^+) \circ (F_1^+ \xrightarrow{n_2 n_1^{-1}} n_2 F_1).$$
$$S( F_2 \xrightarrow{\delta} F_1 ) = (n_2 F_1 \xrightarrow{n_1 n_2^{-1}} F_1^+) \circ (F_2^+ \xrightarrow{\delta} n_2 F_1).$$
\item
If $F \in \Facet$ and $g \in G$, then $F^+ = n_F F \in \Facet^+$ and $(gF)^+ = n_{(gF)} gF \in \Facet^+$.  Define the functor $S$ on $G$-arrows by
$$S(F \xrightarrow{g} gF) = (F^+ \xrightarrow{n_{gF} \cdot g \cdot n_F^{-1}}  (gF)^+).$$
\end{enumerate}

In this way, $S$ defines a functor from the free category on the directed graph with vertex set $\Facet$, up-arrows, down-arrows, and $G$-arrows, to the category $\Cat{Fac}_G^+$.  To see that this defines a functor from $\Cat{Fac}_G$ to $\Cat{Fac}_G^+$, we must check that relations (A)-(G) are respected.

To begin, if $F_1 \leq F_2 \leq F_3$, and the notation is as above,
\begin{align*}
S(F_2 \xrightarrow{\gamma} F_3) \circ S(F_1 \xrightarrow{\gamma} F_2) &= (n_3 F_2 \xrightarrow{\gamma} F_3^+) \circ (F_2^+ \xrightarrow{n_3 n_2^{-1}} n_3 F_2) \\
& \phantom{=} \circ (n_2 F_1 \xrightarrow{\gamma} F_2^+) \circ (F_1^+ \xrightarrow{n_2 n_1^{-1}} n_2 F_1), \\
&= (n_3 F_2 \xrightarrow{\gamma} F_3^+) \circ  (n_3 F_1 \xrightarrow{\gamma} n_3 F_2)  \\
& \phantom{=} \circ (n_2 F_1 \xrightarrow{n_3 n_2^{-1}} n_3 F_1) \circ (F_1^+ \xrightarrow{n_2 n_1^{-1}} n_2 F_1), \quad \text{(by F)} \\
&= (n_3 F_1 \xrightarrow{\gamma} F_3^+) \circ (F_1^+ \xrightarrow{n_3 n_1^{-1}} n_3 F_1), \quad \text{ (by B, E)} \\
&= S(F_1 \xrightarrow{\gamma} F_3).
\end{align*}
Also, we have
\begin{align*}
S(\Id_F) &= S(F \xrightarrow{\gamma} F) = (F^+ \xrightarrow{\gamma} F^+) \circ (F^+ \xrightarrow{n_F n_F^{-1}} F^+), \\
&= \Id_{F^+} \circ \Id_{F^+}, \quad \text{(by A, D)} \\
&= \Id_{F^+}
\end{align*}
For down-arrows, the computations are very similar.  This verifies that relations (A), (B), (C) are respected by $S$.  Relation (D) is trivial to check.  For relation (E), we have
\begin{align*}
S(hF \xrightarrow{g} ghF) \circ S(F \xrightarrow{h} hF) &= ((hF)^+ \xrightarrow{n_{gh F} \cdot g \cdot n_{h F}^{-1}} (ghF)^+) \\
& \phantom{=} \circ (F^+ \xrightarrow{n_{h F} \cdot h \cdot n_{F}^{-1}} (hF)^+), \\
&= (F^+ \xrightarrow{ n_{gh F} \cdot g h \cdot n_F^{-1}  } (ghF)^+ ), \\
&= S( F \xrightarrow{gh} ghF ).
\end{align*}
For relation (F), the intertwining of $G$-arrows and up-arrows, we have
\begin{align*}
S( F_2 \xrightarrow{g} g F_2) \circ S(F_1 \xrightarrow{\gamma} F_2) &= (F_2^+ \xrightarrow{ n_{g F_2} \cdot g \cdot n_{F_2}^{-1} } (g F_2)^+) \circ (n_{F_2} F_1 \xrightarrow{\gamma} F_2^+) \\
& \phantom{=} \circ (F_1^+ \xrightarrow{n_{F_2} \cdot n_{F_1}^{-1}} n_{F_2} F_1), \\
&= (n_{g F_2} g F_1 \xrightarrow{\gamma} (g F_2)^+) \circ (n_{F_2} F_1 \xrightarrow{ n_{g F_2} \cdot g \cdot n_{F_2}^{-1} } n_{g F_2} g F_1) \\
& \phantom{=}  \circ (F_1^+ \xrightarrow{n_{F_2} \cdot n_{F_1}^{-1}} n_{F_2} F_1), \\
&= (n_{g F_2} g F_1 \xrightarrow{\gamma} (g F_2)^+) \circ ( (g F_1)^+ \xrightarrow{ n_{g F_2} n_{g F_1}^{-1} } n_{g F_2} g F_1) \\
& \phantom{=} \circ (F_1^+ \xrightarrow{ n_{g F_1} \cdot g \cdot n_{F_1}^{-1} } (g F_1)^+ ), \quad \text{ (by E)} \\
&= S(g F_1 \xrightarrow{\gamma} g F_2) \circ S( F_1 \xrightarrow{g} g F_1).
\end{align*}
For the intertwining of $G$-arrows and down-arrows, the computations are similar.  This verifies that relations (F) and (G) are respected by $S$, and thus we have defined a functor as claimed,
$$S \From \Cat{Fac}_G \To \Cat{Fac}_G^+.$$
By construction, we have $S \circ \Inc = \Id$, as functors from $\Cat{Fac}_G^+$ to itself.  It remains to see that $\Inc \circ S$ is naturally isomorphic to the identity functor.  Note that $[\Inc \circ S](F) = F^+$ for all objects $F$ of $\Cat{Fac}_G$.  Define a natural isomorphism $N \From \Id \Rightarrow \Inc \circ S$ by putting
$$N(F) = (F \xrightarrow{n_F} F^+).$$
To see that this defines a natural isomorphism, we check the commutativity of diagrams in the category $\Cat{Fac}_G$.
$$\begin{tikzcd}
F_1 \arrow{r}{n_{F_1}} \arrow{d}{\gamma} & F_1^+ \arrow{d}{\gamma \circ n_{F_2} n_{F_1}^{-1}} \\
F_2 \arrow{r}{n_{F_2}} & F_2^+
\end{tikzcd}
\quad
\begin{tikzcd}
F_1 \arrow{r}{n_{F_1}}  & F_1^+  \\
F_2 \arrow{r}{n_{F_2}} \arrow{u}{\delta} & F_2^+ \arrow{u}[swap]{\delta \circ n_{F_1} n_{F_2}^{-1}}
\end{tikzcd}
$$
Commutativity of the left diagram follows from relation (E) and (F), and the right diagram from (E) and (G).

Finally, we have to check the commutativity of the diagram below.
$$\begin{tikzcd}
F \arrow{d}{g} \arrow{r}{n_F} & F^+ \arrow{d}{n_{g F} \cdot g \cdot n_F^{-1}} \\
gF \arrow{r}{n_{gF}} & (gF)^+
\end{tikzcd}$$
This is obvious, verifying that $N$ is a natural isomorphism.
\end{proof}

\subsection{Monotonic equivariant bisheaves}

A bisheaf $(\sheaf{V}, \gamma, \delta)$ on any poset $\Facet$ is called \defined{monotonic} if for all facets $F_1 \leq F_2$,
$$\Id_{\sheaf{V}_{F_2}} = \gamma_{F_1, F_2} \circ \delta_{F_2, F_1}   \From \sheaf{V}_{F_2} \To \sheaf{V}_{F_1} \To \sheaf{V}_{F_2}.$$
Within the category of bisheaves, write $\Cat{BiSh}^{\mon}(\Facet, \Omega)$ for the full subcategory of $\Cat{BiSh}(\Facet, \Omega)$ whose objects are the monotonic bisheaves.  Similarly, write $\Cat{BiSh}_G^{\mon}(\Facet, \Omega)$ for the category of $G$-equivariant monotonic bisheaves.

\begin{proposition}
\label{monores}
Suppose that $\Facet^+$ is a $G$-full subposet of a poset $\Facet$ with $G$-action.  Then the restriction functor,
$$\Res \From \Cat{BiSh}_G(\Facet, \Omega) \xrightarrow{\sim} \Cat{BiSh}_G(\Facet^+, \Omega),$$
restricts to an equivalence of categories,
$$\Res \From \Cat{BiSh}_G^{\mon}(\Facet, \Omega) \xrightarrow{\sim} \Cat{BiSh}_G^{\mon}(\Facet^+, \Omega).$$
\end{proposition}
\begin{proof}
The restriction functor sends monotonic bisheaves to monotonic bisheaves.  The only thing to verify is that the inverse functor sends monotonic bisheaves to monotonic bisheaves.  For this, we recall the functor $S \From \Cat{Fac}_G \To \Cat{Fac}_G^+$ (see the proof of Theorem \ref{restful}) and compute,
\begin{align*}
S(F_1 \xrightarrow{\gamma} F_2) \circ S(F_2 \xrightarrow{\delta} F_1) &= 
(n_2 F_1 \xrightarrow{\gamma} F_2^+) \circ (F_1^+ \xrightarrow{n_2 n_1^{-1}} n_2 F_1) \\
&\phantom{=} \circ (n_2 F_1 \xrightarrow{n_1 n_2^{-1}} F_1^+) \circ (F_2^+ \xrightarrow{\delta} n_2 F_1), \\
&= (n_2 F_1 \xrightarrow{\gamma} F_2^+) \circ (F_2^+  \xrightarrow{\delta} n_2 F_1).  \quad \text{(by (E))}
\end{align*}
If $\sheaf{V}$ is a monotonic bisheaf on $\Facet^+$, viewed as a functor $\Cat{Fac}_G^+ \To \Omega$, then $\sheaf{V}(\gamma_{F_1, F_2} \circ \delta_{F_2, F_1}) = \Id_{\sheaf{V}_{F_2}}$.  The above computation shows that
$$[\sheaf{V} \circ S](F_1 \xrightarrow{\gamma} F_2 \xrightarrow{\delta} F_1) = \sheaf{V}(n_2 F_1 \xrightarrow{\gamma} F_2^+ \xrightarrow{\delta} n_2 F_1) = \Id.$$
Hence $\sheaf{V} \circ S = S^\ast \sheaf{V}$ is a monotonic bisheaf as well.
\end{proof}

\subsection{Fundamental domains}

We may strengthen the ``$G$-fullness'' condition, to define a notion of fundamental domain.
\begin{definition}
A subset $\Facet^+ \subset \Facet$ is called a \defined{closed fundamental domain} if it satisfies the following conditions.
\begin{enumerate}
\item
(closedness):  if $F,F' \in \Facet$, $F \leq F'$, and $F' \in \Facet^+$, then $F \in \Facet^+$.
\item
If $F \in \Facet$, then $\{ g F : g \in G \} \cap \Facet^+$ has cardinality one.
\end{enumerate}
\end{definition}

Every closed fundamental domain in $\Facet$ is also a $G$-full poset.  When $\Facet^+$ is a closed fundamental domain, the resulting category $\Cat{Fac}_G^+$ is particularly nice.  It is the category generated by up-arrows and down-arrows, as before, together with the following $G$-arrows:  for every $F \in \Facet^+$, let $G_F$ be the stabilizer of $F$ in $G$.  Then the $G$-arrows generating $\Cat{Fac}_G^+$ are precisely the arrows $F \xrightarrow{g} F$, for $F \in \Facet^+$ and $g \in G_F$.  In this way, the category $\Cat{Fac}_G^+$ can be thought of as a ``stacky quotient'' $[\Cat{Fac} / G]$.  From this perspective, it may not be surprising that $G$-equivariant bisheaves on $\Facet$ can be interpreted as bisheaves on this stacky quotient $[\Cat{Fac} / G]$.

In what follows, we apply this to the case of Coxeter arrangements, reducing bisheaves on the reflection representation of a Coxeter group $W$ to objects on a closed chamber.

\section{Coxeter arrangements}
Our treatment of Coxeter arrangements, for finite or affine Coxeter groups, follows Bourbaki \cite[Chapter 5, \S 3]{Bour}.  Let $E$ be a real affine space of dimension $d$, whose space of translations is endowed with a Euclidean metric.  Let $\Hyper$ be a set of hyperplanes in $E$, and $W$ the group of affine-linear automorphisms of $E$ generated by reflections across these hyperplanes.  Assume the two conditions, called (D1), (D2) in Bourbaki,
\begin{enumerate}
\item[(D1)] For every $w \in W$ and $H \in \Hyper$, the hyperplane $w(H)$ belongs to $\Hyper$.
\item[(D2)] If $K$ and $K'$ are compact subsets of $E$, then $\{ w \in W : w(K) \cap K' \neq \emptyset \}$ is finite.
\end{enumerate}
It follows that the set of hyperplanes is locally finite in $E$.  The set of hyperplanes $\Hyper$ gives an equivalence relation on $E$, by declaring $x \sim y$ if for all $H \in \Hyper$, either $x,y \in H$ or $x,y$ lie on the same side of $H$.  The equivalence classes for this relation are called \defined{facets}, and we write $\Facet$ for the set of facets.  We have $E = \bigsqcup_{F \in \Facet} F$.  A \defined{chamber} is a facet which is contained in no hyperplane from $\Hyper$.  

If $F_1, F_2 \in \Facet$, we write $F_1 \leq F_2$ if $F_1$ is contained in the closure of $F_2$.  Every facet is contained in the closure of some chamber.  Every facet $F$ has a well-defined dimension -- the dimension of the minimal affine subspace of $E$ containing $F$.  The chambers are the facets of dimension $d = \dim(E)$.

Let $C$ be a chamber.  A facet $F \leq C$ is called a \defined{face} of $C$ if $\dim(F) = d-1$.  A hyperplane $H \in \Hyper$ is called a \defined{wall} of $C$ if there exists a face $F \leq C$ such that $F \subset H$.  Let $S$ be the set of reflections with respect to the walls of $C$.  Then in \cite[Chapter 5, \S 3.2]{Bour}, it is verified that $(W,S)$ is a Coxeter system.  This means
\begin{enumerate}
\item
$S$ is a finite set which generates $W$.
\item
For all $s \in S$, we have $s^2 = 1$.
\item
For any pair $s, s' \in S$, write $m_{s,s'}$ for the order of $s s'$ in $W$.  Then a presentation of $W$ is given by the set of generators $S$ and the set of relations $\{ (s s')^{m_{s,s'}} : m_{s,s'} < \infty \}$.  
\end{enumerate}
For this reason, we call $(E, \Hyper)$ a \defined{Coxeter arrangement}.  When $(E, \Hyper)$ is a Coxeter arrangement, and $C$ is a chamber, we have constructed a Coxeter system $(W,S)$ and a poset $\Facet$ endowed with a $W$-action.

\subsection{Restriction of bisheaves to a chamber}
Fix a Coxeter arrangement $(E, \Hyper)$ as before, with facet set $\Facet$, a fixed chamber $C$, and the resulting Coxeter system $(W,S)$.  Let $\Facet^+ = \{ F \in \Facet : F \leq C \}$, a fundamental domain for $W$ acting on $\Facet$.  The partially ordered set $\Facet^+$ has an alternative description in terms of $S$.

Define $\Lambda$ to be the set of all finitary subsets of $S$.  Here a \defined{finitary} subset of $S$ is a subset $I \subset S$ which generates a finite subgroup of $W$.  Thus if $W$ is finite, $\Lambda$ is the set of all subsets of $S$.  We make $\Lambda$ a poset by {\em reverse-inclusion} (so $I \leq J$ means that $I \supset J$).
\begin{proposition}
If $I \in \Lambda$, define 
$$F_I = \{ x \in \bar C : s(x) = x \text{ for all } s \in I, \text{ and } s(x) \neq x \text{ for all } s \in S - I \}.$$
Then the map $I \mapsto F_I$ gives a poset isomorphism from $\Lambda$ to $\Facet^+$.
\end{proposition}
\begin{proof}
For injectivity, suppose that $I,J \in \Lambda$ and $I \neq J$.  Note that the conditions defining $F_I$ and those defining $F_J$ are mutually exclusive; thus $F_I \neq F_J$.  For surjectivity, consider a facet $F \leq C$.  Let $I$ be the set of walls containing it, identified with a subset of $S$ (identifying a wall with the reflection across it).  Then the closure $\bar F$ contains at least one 0-dimensional facet $v \in \Facet^+$.  It follows that $I$ is contained in the set $I_v$ of walls containing $v$.  But the reflections through walls containing $v$ generate a finite Coxeter group, and so $I_v$ is finitary.  Hence $I$ is finitary.
\end{proof}

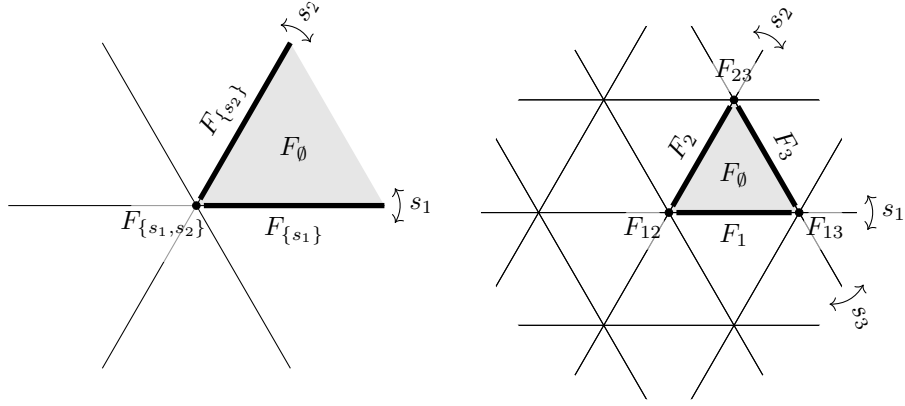
\begin{figure}[htbp!]
\begin{tikzpicture}
\clip (0,0) circle (3.1);
\fill[black!10] (0:2.5) -- (0,0) -- (60:2.5);
\foreach \theta in {0, 60, 120, 180, 240 , 300}
{
\draw (0,0) -- (\theta: 2.5);
}
\draw (30:1.5) node {$F_\emptyset$};
\draw[<->] (-5:2.65) to [bend right = 30] node[right] {$s_1$} (5:2.65);
\draw[<->] (55:2.65) to [bend right = 30] node[right, rotate=60] {$s_2$} (65:2.65);
\draw[line width=2pt] (0.1,0) to node[below] {$F_{\{ s_1 \}}$} (0:2.5);

\draw[line width=2pt] (60:0.1) to node[above, rotate=60] {$F_{\{ s_2 \}}$} (60:2.5);

\fill[white, opacity=0.6] (210:0.5) circle (0.5);
\filldraw (0,0) circle (0.05);
\draw (210:0.5) node {$F_{\{ s_1, s_2\} }$};

\end{tikzpicture}
\hfill
\begin{tikzpicture}

\draw[white] (0,0) circle (3.0);
\fill[black!10] (0:1.75) -- (0,0) -- (60:1.75);
\draw[<->] (-5:2.65) to [bend right = 30] node[right] {$s_1$} (5:2.65);
\draw[<->] (55:2.65) to [bend right = 30] node[right, rotate=60] {$s_2$} (65:2.65);
\draw (30:1) node {$F_\emptyset$};
\draw[line width = 2pt] (60:0.1) to node[above, rotate=60] {$F_2$} (60:1.632);
\draw[line width = 2pt] (0:0.1) to node[below] {$F_1$} (0:1.632);

\begin{scope}[xshift=0.866 cm, yshift = 1.5cm ]
\draw[<->] (-55:3) to [bend left = 30] node[right, rotate=-60] {$s_3$} (-65:3);
\draw[line width = 2pt] (-60:0.1) to node[above, rotate=-60] {$F_3$} (-60:1.632);
\end{scope}
\begin{scope}
\clip (0,0) circle (2.5);
\foreach \theta in {0, 60, 120, 180, 240 , 300}
{
\begin{scope}[rotate = \theta]
\foreach \off in {-3, -1.5, 0, 1.5, 3}
{
\begin{scope}[yshift = \off cm]
\draw (-2.5,0) -- (2.5,0);
\end{scope}
}
\end{scope}

}
\end{scope}
\filldraw (0,0) circle (0.05);
\fill[white, opacity=0.6] (210:0.4) circle (0.3);
\draw (210:0.4) node {$F_{12}$};

\filldraw (0.866,1.5) circle (0.05);
\fill[white, opacity=0.6] (0.866,1.9) circle (0.3);
\draw (0.866,1.9) node {$F_{23}$};

\begin{scope}[xshift = 1.732cm]
\filldraw (0,0) circle (0.05);
\fill[white, opacity=0.6] (-30:0.4) circle (0.3);
\draw (-30:0.4) node {$F_{13}$};
\end{scope}
\end{tikzpicture}

\caption{Coxeter arrangements of type $A_2$ and $\hat A_2$, with fixed chambers.  Facets in the closure of the chamber are labeled by subsets of $S$, with the chamber corresponding to $\emptyset$.  On the right, subsets are abbreviated by their numbers, e.g., $F_{12}$ stands for $F_{\{ s_1, s_2 \} }$.}
\end{figure}

Restriction of bisheaves, together with the above proposition, gives a functor,
$$\Res \From \Cat{BiSh}(\Facet, \Omega) \To \Cat{BiSh}(\Facet^+, \Omega) \xrightarrow{\sim} \Cat{BiSh}(\Lambda, \Omega).$$
The second functor is an {\em isomorphism} of categories.  To make the category $\Cat{BiSh}(\Lambda, \Omega)$ concrete, a $\Omega$-valued bisheaf on $\Lambda$ consists of objects $\sheaf{V}_I \in \Omega$ for each $I \in \Lambda$, together with morphisms
$$\gamma_{I,J} \From \sheaf{V}_I \To \sheaf{V}_J, \quad \delta_{J,I} \From \sheaf{V}_J \To \sheaf{V}_I,  \text{ for all } I \supset J,$$
satisfying the axioms
$$\gamma_{J,K} \circ \gamma_{I,J} = \gamma_{I,K}, \quad \delta_{J,I} \circ \delta_{K,J} = \delta_{K,I} \text{ if } I \supset J \supset K,$$
and $\delta_{I,I} = \gamma_{I,I} = \Id_{\sheaf{V}_I}$ for all $I \in \Lambda$.
A bisheaf on $\Lambda$, valued in $\Cat{Vec}_\CC$ (for example), is the same as a representation of the {\em Hasse quiver} discussed in \cite[\S 5]{DJS}.

Applying Theorem \ref{restful} and Proposition \ref{monores}, we have
\begin{thm}
Restriction to a chamber gives an equivalence of categories,
$$\Cat{BiSh}_W(\Facet, \Omega) \xrightarrow{\sim} \Cat{BiSh}_W(\Lambda, \Omega),$$
from $W$-equivariant bisheaves on the Coxeter arrangement to $W$-equivariant bisheaves on the poset $\Lambda$ of finitary subsets of $S$.  This equivalence respects the subcategories of monotonic bisheaves.
\end{thm}

It is important to note that the closed Weyl chamber $\Facet^+$ is a closed fundamental domain; it follows that a $W$-equivariant structure for a sheaf on $\Lambda$ (equivalently on $\Facet^+$) is given by actions of stabilizing subgroups only; the $W$-arrows of the category $\Cat{Fac}_W^+$ are precisely the arrows $I \xrightarrow{w} I$ for all $I \in \Lambda$ and $w \in W_I$.

\subsection{Modules and algebras}

Let $R$ be a (commutative, unital) ring, and $\Omega$ the category of $R$-modules.  Suppose that $(W,S)$ is a {\em finite} Coxeter system, so that $\Lambda$ is the power set of $S$.  In this case, the category $\Cat{BiSh}_W^{\mon}(\Lambda, \Omega)$ has a concrete interpretation.

Let $(\sheaf{V}, \gamma, \delta, \eta)$ be a monotonic, $W$-equivariant $\Omega$-valued bisheaf on $\Lambda$.  Thus, for every $I \in \Lambda$, we have an $R$-module $\sheaf{V}_I$.  The condition of monotonicity states that
$$\gamma_{I,J} \circ \delta_{J,I} = \Id \From \sheaf{V}_J \To \sheaf{V}_J \text{ for all } I \supset J.$$
It follows that $\delta_{J,I}$ is injective and $\gamma_{I,J}$ is surjective.

Define the following data from $(\sheaf{V}, \gamma, \delta, \eta)$:
\begin{itemize}
\item
Let $V = \sheaf{V}_S$, the $R$-module associated to the unique vertex in $\bar C$.
\item
For all $I \subset S$, define $e_I = \delta_{I,S} \circ \gamma_{S,I} \in \End(V)$.  Monotonicity implies that $e_I$ is idempotent.
\item
For $w \in W$, write $w$ for the automorphism $\eta(w) \in \Aut(V) = \Aut(\sheaf{V}_S)$ arising from the equivariant structure on $\sheaf{V}$.
\end{itemize}

The equivariant bisheaf axioms imply the following properties of the idempotents $e_I$ and automorphisms $\eta(w)$.
\begin{enumerate}
\item
If $I \supset J$ then $e_I e_J = e_J$.
\item
If $s \in I$, then $s e_I = e_I s$.
\end{enumerate}
Motivated by this, let $\AA_W^{\mon}$ be the $R$-algebra $R \langle e_I : I \in \Lambda,\ s : s \in S \rangle$ modulo the relations
\begin{enumerate}
\item
If $I \supset J$ then $e_I e_J = e_J$.
\item
If $s \in I$ then $s e_I = e_I s$.
\item
$s^2 = 1$ for all $s \in S$.
\item
$(s s')^{m_{s,s'}} = 1$ for all $s,s' \in S$ with $m_{s,s'} < \infty$.
\end{enumerate}

\begin{proposition}
\label{monequiv}
The construction above gives an equivalence of categories, from $\Cat{BiSh}_W^{\mon}(\Lambda, \Omega)$ to the category of modules over $\AA_W^{\mon}$ 
\end{proposition}
\begin{proof}
The construction described above gives a functor $(\sheaf{V}, \gamma, \delta, \eta) \mapsto V$ from $\Cat{BiSh}_W^{\mon}(\Lambda, \Omega)$ to $\Cat{Mod}(\AA_W^{\mon})$.

From an $\AA_W^{\mon}$-module $V$, one may reconstruct an equivariant bisheaf $(\sheaf{V}, \gamma, \delta, \eta)$ according to the following recipe.  Write $\sheaf{V}_I = e_I V$.  Define $\delta_{J,I} \From \sheaf{V}_J \To \sheaf{V}_I$ by $\delta(v) = v$, whenever $I \supset J$, i.e., $\delta_{J,I}$ is the inclusion map.  Define $\gamma_{I,J} \From \sheaf{V}_I \To \sheaf{V}_J$ by $\gamma_{I,J}(v) = e_J v$.  The bisheaf conditions are clear, as they follow from the conditions $e_I e_J = e_J$ for $I \supset J$.  We have constructed a bisheaf $(\sheaf{V}, \gamma, \delta)$ on the poset $\Lambda$.

The relations (3) and (4) in $\AA_W^{\mon}$ are precisely the Coxeter relations in $W$, and so every $\AA_W^{\mon}$-module becomes a $W$-module.  In this way, $\sheaf{V}_S = V$ inherits ``$W$-morphisms'' -- it is a $RW$-module.  The commutativity $s e_I = e_I s$ implies that every submodule $\sheaf{V}_I = e_I V$ inherits the structure of a $W_I$-module.  This is precisely the data of a $W$-equivariant structure on the bisheaf $\sheaf{V}$.

This construction gives an inverse functor $\Cat{Mod}(\AA_W^{\mon}) \To \Cat{BiSh}_W^{\mon}(\Lambda, \Omega)$.  
\end{proof}

\section{Perverse sheaves after Kapranov and Schechtman}

Let $\Hyper$ be a locally finite hyperplane arrangement in a real affine space $E$.  Let $\Facet$ be the resulting partially ordered set of facets.  Let $\Hyper_\CC$ and $E_\CC$ be the complexification; then $\Hyper_\CC$ defines a stratification of the complex affine space $E_\CC$.  In \cite{KapSch}, Kapranov and Schechtman describe the category $\Cat{Perv}(E_\CC, \Hyper_\CC)$ of perverse sheaves on $E_\CC$, smooth with respect to this stratification.  Specifically, Kapranov and Schechtman define an equivalence of categories from $\Cat{Perv}(E_\CC, \Hyper_\CC)$ to a full subcategory of $\Cat{BiSh}(\Facet, \Omega)$, where $\Omega$ is the category of finite-dimensional complex vector spaces.  This full subcategory consists of those bisheaves which are {\em monotonic} and {\em transitive} and {\em invertible}.

We have discussed the monotonic axiom, and in this section we review the axioms for transitivity and invertibility.  These notions rely on more than the poset structure of $\Facet$ -- the structure of the real hyperplane arrangement plays a deeper role.   

\subsection{The transitive axiom}

Fix $\Hyper$ and $E$ as above, a locally finite real hyperplane arrangement, with resulting facet decomposition $\Facet$.  If $F \in \Facet$, the \defined{star} of $F$ is defined by 
$$\Star(F) = \bigcup_{F' : F \leq F'} F'.$$
Then $\Star(F)$ is an open neighborhood of $F$ in $E$.  Following \cite[\S 9.B]{KapSch}, we define \defined{collinear facets} in such a star-neighborhood.
\begin{definition}
Let $F_0$ be a facet, and suppose that $(F_1, F_2, F_3)$ is an ordered triple of facets in the star $\Star(F_0)$.  We say this triple is \defined{collinear} if there exist points $p_1, p_2, p_3$ in $F_1, F_2, F_3$, respectively, such that $p_2 \in \overline{p_1 p_3}$ (the line segment joining $p_1$ and $p_3$ in $\Apart$).
\end{definition}

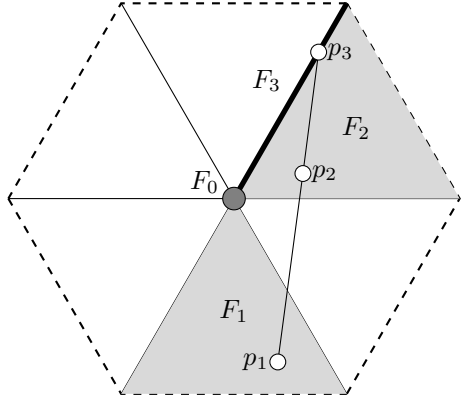
\begin{figure}[htbp!]
\begin{tikzpicture}[scale=1.5]
\foreach \t in {0,120,240}
{
\draw (\t:-2) -- (\t:2);
\draw[dashed, thick] (\t - 60:2) -- (\t:2) -- (\t+60:2);
}
\fill[fill=gray!30] (0,0) -- (-60:2) -- (-120:2) -- cycle;
\fill[fill=gray!30] (0,0) -- (60:2) -- (0:2) -- cycle;
\filldraw[fill=gray, line width = 2pt] (0,0) -- (60:2);
\filldraw[fill=gray, draw=black] (0,0) circle (0.1);
\draw (150:0.3) node {$F_0$};

\node (F1) at (-90:1) {$F_1$};
\node (F2) at (30:1.25) {$F_2$};
\node[above left] (F3) at (60:1) {$F_3$};

\coordinate (p1) at (-75:1.5); 
\coordinate (p2) at (20:0.65); 
\coordinate (p3) at (60:1.5);
\draw (p1) -- (p2) -- (p3);
\filldraw[draw=black, fill=white] (p1) circle (2pt) node[left] {$p_1$};
\filldraw[draw=black, fill=white] (p2) circle (2pt) node[right] {$p_2$};
\filldraw[draw=black, fill=white] (p3) circle (2pt) node[right] {$p_3$};
\end{tikzpicture}
\caption{The star neighborhood $\Star(F_0)$.  The triple of facets $(F_1, F_2, F_3)$ (of dimensions $2,2,1$, respectively) are collinear, exhibited by the line segment through $p_1, p_2, p_3$.}
\end{figure}

\begin{definition}
A bisheaf $(\sheaf{V}, \gamma, \delta)$ on $\Facet$ is called \defined{transitive} if for all $F_0 \in \Facet$ and all collinear triples $(F_1, F_2, F_3)$ in the star neighborhood $\Star(F_0)$, we have
$$\gamma_{F_0,F_3} \delta_{F_2, F_0} \circ \gamma_{F_0, F_2} \delta_{F_1, F_0} = \gamma_{F_0, F_3} \delta_{F_1, F_0}.$$
I.e., both paths in the diagram below lead to the same map from $\sheaf{V}_{F_1}$ to $\sheaf{V}_{F_3}$.
$$\begin{tikzcd}
\sheaf{V}_{F_1} \arrow{dr}{\delta} & & \sheaf{V}_{F_2} \arrow{dr}{\delta} & & \sheaf{V}_{F_3} \\
& \sheaf{V}_{F_0} \arrow{ur}{\gamma} \arrow{rr}{=} & & \sheaf{V}_{F_0} \arrow{ur}{\gamma}& 
\end{tikzcd}$$
\end{definition}

Following the notation of \cite{KapSch}, we define $\phi_{A,B} = \gamma_{F_0, B} \circ \delta_{A,F_0}$ when $A,B$ are contained in the star neighborhood $\Star(F_0)$, and $F_0$ is held fixed.  In this notation, the transitive axiom becomes
$$\phi_{F_2, F_3} \circ \phi_{F_1, F_2} = \phi_{F_1, F_3}.$$

Consider a line segment in $E$, and the sequence of facets it passes through.  Each time the segment passes from a facet $F$ to a facet $F'$, we have $\dim(F) < \dim(F')$ or $\dim(F) > \dim(F')$; we say the segment goes Up or Down accordingly.  Since hyperplanes are determined by affine-linear functions on $E$, one cannot have a segment go Up then Up again, or Down then Down again.  The dimensions of facets must alternately increase and decrease along the segment.

For segments that follow a Down-Up pattern, the axiom of monotonicity implies the axiom of transitivity.
\begin{lemma}
\label{DUseg}
Suppose $F_0 \in \Facet$, and $F_1, F_2, F_3$ are facets in $\Star(F_0)$.  Assume moreover that $F_1 \geq F_2$ and $F_2 \leq F_3$.  If $\sheaf{V}$ is a monotonic bisheaf on $\Facet$, then $\phi_{F_2, F_3} \circ \phi_{F_1, F_2} = \phi_{F_1, F_3}.$
\end{lemma}
\begin{proof}
If $F_1 \geq F_2$ and $F_2 \leq F_3$, then we find the diagram of linear maps below.
$$\begin{tikzcd}
\sheaf{V}_{F_1} \arrow{ddr}[swap]{\delta}  \arrow{dr}{\delta} & & & & \sheaf{V}_{F_3} \\
& \sheaf{V}_{F_2} \arrow{d}{\delta} \arrow{r}{=} & \sheaf{V}_{F_2} \arrow{dr}{\delta} \arrow{r}{=} & \sheaf{V}_{F_2} \arrow{ur}{\gamma} &  \\
& \sheaf{V}_{F_0} \arrow{ur}{\gamma} \arrow{rr}{=} & & \sheaf{V}_{F_0} \arrow{u}{\gamma} \arrow{uur}[swap]{\gamma}& 
\end{tikzcd}$$
The bisheaf axioms and monotonicity give equality of the following linear maps from $\sheaf{V}_{F_1}$ to $\sheaf{V}_{F_3}$:
\begin{align*}
&\phantom{=} \sheaf{V}_{F_1} \xrightarrow{\delta} \sheaf{V}_{F_0} \xrightarrow{=} \sheaf{V}_{F_0} \xrightarrow{\gamma} \sheaf{V}_{F_3} \\
& = \sheaf{V}_{F_1} \xrightarrow{\delta} \sheaf{V}_{F_2} \xrightarrow{\delta} \sheaf{V}_{F_0} \xrightarrow{\gamma} \sheaf{V}_{F_2} \xrightarrow{\gamma} \sheaf{V}_{F_3}, \\ 
& = \sheaf{V}_{F_1} \xrightarrow{\delta} \sheaf{V}_{F_2} \xrightarrow{=} \sheaf{V}_{F_2} \xrightarrow{\gamma} \sheaf{V}_{F_3} \quad \text{(by monotonicity)},\\  
&= \sheaf{V}_{F_1} \xrightarrow{\delta} \sheaf{V}_{F_2} \xrightarrow{\delta} \sheaf{V}_{F_0} \xrightarrow{\gamma} \sheaf{V}_{F_2} \xrightarrow{\delta} \sheaf{V}_{F_0} \xrightarrow{\gamma} \sheaf{V}_{F_2} \xrightarrow{\gamma} \sheaf{V}_{F_3} \quad \text{(by monotonicity)},\\  
&= \sheaf{V}_{F_1} \xrightarrow{\delta} \sheaf{V}_{F_0} \xrightarrow{\gamma} \sheaf{V}_{F_2} \xrightarrow{\delta} \sheaf{V}_{F_0}  \xrightarrow{\gamma} \sheaf{V}_{F_3}.\\ 
\end{align*}
This verifies the identity.
\end{proof}

If $(F_1, F_2, F_3)$ is a collinear triple in a star neighborhood $\Star(F_0)$, then we say $F_2$ is in \defined{general position} if the dimension of $F_2$ is maximal among those facets collinear between $F_1$ and $F_3$.  

\begin{lemma}
\label{transgen}
A monotonic bisheaf $\sheaf{V}$ on $\Facet$ is transitive if and only if the transitivity axiom holds for collinear triples $(F_1, F_2, F_3)$ with $F_2$ in general position.
\end{lemma}
\begin{proof}
Consider any collinear triple $(F_1, F_2, F_3)$ in a star neighborhood $\Star(F_0)$.  If $F_1 = F_2$ or $F_2 = F_3$, then the transitivity condition follows from the monotonicity condition.  So we are left to consider collinear triples with $F_1 \neq F_2$ and $F_2 \neq F_3$.  

Let $\lambda$ be a line in $E$, continuously and injectively parameterized so that $\lambda(1) \in F_1$ and $\lambda(2) \in F_2$ and $\lambda(3) \in F_3$.  Then, for sufficiently small and positive $\epsilon$, $\lambda(2 - \epsilon)$ and $\lambda(2 + \epsilon)$ belong to facets $F_2^-$ and $F_2^+$, respectively, and $F_2 \leq F_2^+$ and $F_2 \leq F_2^-$.  Moreover, $F_2^-$ and $F_2^+$ are facets in general position (among those collinear between $F_1$ and $F_3$).

Now we may use the transitivity axiom for triples whose middle-term is in general position, in order to prove the transitivity axiom for $(F_1, F_2, F_3)$.  \begin{align*}
\phi_{F_2, F_3} \circ \phi_{F_1, F_2} &= \phi_{F_2^+, F_3} \circ \phi_{F_2, F_2^+} \circ \phi_{F_2^-, F_2} \circ \phi_{F_1, F_2^-}, \text{\quad (by general position of $F_2^\pm$)} \\
&= \phi_{F_2^+, F_3} \circ  \phi_{F_2^-, F_2^+} \circ \phi_{F_1, F_2^-}, \quad \text{ (by Lemma \ref{DUseg})} \\
&= \phi_{F_2^-, F_3} \circ \phi_{F_1, F_2^-}, \quad \text{(by general position of $F_2^+$)} \\
&= \phi_{F_1, F_3}. \quad \text{(by general position of $F_2^-$)} 
\end{align*}
\end{proof}

\subsection{The invertible axiom}

As before, fix a locally finite real hyperplane arrangement $\Hyper \subset E$, with resulting facet decomposition $\Facet$.  A \defined{flat} in $E$ is a nonempty intersection of hyperplanes from $\Hyper$.  If $F \in \Facet$, and $L$ is a flat, we say that $F$ spans $L$ if $L$ is the intersection of all hyperplanes containing $F$.  Note that $\dim(L) = \dim(F)$.  

\begin{definition}
Let $F_1, F_2$ be facets, with $\dim(F_1) = \dim(F_2) = r \geq 1$.  We say that $F_1$ \defined{opposes} $F_2$ if they span the same flat $L$, and there exists a $(r-1)$-dimensional facet $F_0$ such that $F_0 \leq F_1$ and $F_0 \leq F_2$, and $F_1, F_2$ lie on opposite sides of $F_0$.  In particular, there exists a hyperplane $H_0$ containing $F_0$, such that $F_1$ and $F_2$ lie on opposite side of $H_0$.
\end{definition}

When $F_1$ opposes $F_2$ in $\Apart$, and $\dim(F_1) = \dim(F_2) = r$, the intersection $\bar F_1 \cap \bar F_2$ contains a unique $(r-1)$-dimensional facet $F_0$.  Thus $F_1$ and $F_2$ uniquely determine $F_0$, when $F_1$ and $F_2$ are opposite.  We say that $F_1$ opposes $F_2$ \defined{through} $F_0$ in this case, and we write $F_1 \mid_{F_0} F_2$.  See Figure \ref{OppFig}.
\begin{figure}[htbp!]

\begin{tikzpicture}[scale=2.5]
\clip (0,0) circle (1.5);
\pgfmathsetmacro{\e}{0.33333};
\foreach \t in {0,120,240}
\foreach \j in {-3,...,3}
{
\begin{scope}[rotate=\t]
\begin{scope}[yshift = \j cm]
\%draw[cyan!50, very thin] (0:-5) -- (0:5);
\filldraw[fill=gray, opacity=0.2] (-5, \e) -- (5,\e) -- (5,-\e) -- (-5,-\e) -- cycle;
\draw (-5, -\e) -- (5,-\e);
\draw (-5, \e) -- (5,\e);
\end{scope}
\end{scope}
}
\draw (0,0) node {$F_1$};
\draw (30:0.45) node {$F_2$};
\draw (60:0.55) node {$F_3$};

\begin{scope}[xshift = 0.1925cm, yshift = -0.34cm]
\draw[line width = 2pt] (-0.77,0) + (60:0.03) -- +(60:0.36) node[below = 4pt] {$A$};
\draw[line width = 2pt] (-0.77,0.008) + (180:0.03) to node[below] {$C$} +(180:0.36);
\draw[line width = 2pt] (-0.573,0.34) + (60:0.03) -- +(60:0.36) node[below = 4pt] {$B$};
\end{scope}
\end{tikzpicture}

\caption{Among the 2-dimensional facets, $F_1$ opposes $F_2$, and $F_2$ opposes $F_3$, but $F_1$ does not oppose $F_3$.  Among the 1-dimensional facets, $A$ opposes $B$, but $A$ does not oppose $C$ (since $A$ and $C$ do not span the same flat).}

\label{OppFig}

\end{figure}
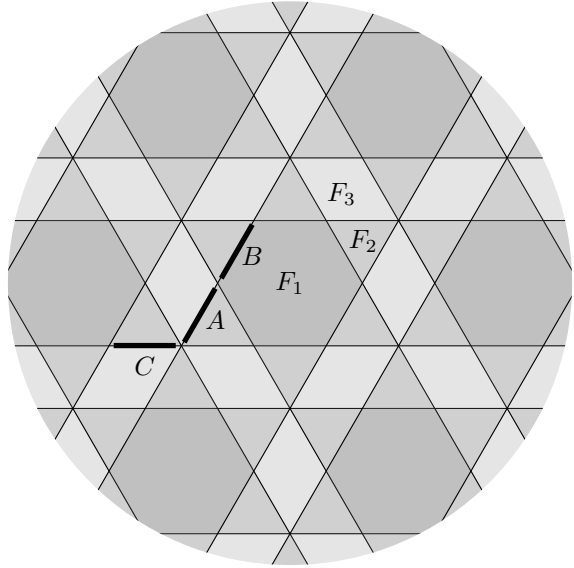

\begin{definition}
A bisheaf $(\sheaf{V}, \gamma, \delta)$ on $\Facet$ is called \defined{invertible} if $F_1 \mid_{F_0} F_2$ implies that the map 
$$\gamma_{F_0,F_2} \circ \delta_{F_1,F_0} \From \sheaf{V}_{F_1} \To \sheaf{V}_{F_2}$$
is an isomorphism.
\end{definition}

\subsection{Perverse sheaves}

Let $\Hyper$ be a locally finite real hyperplane arrangement in the real affine space $E$.  Let $\Facet$ be the resulting facet decomposition.  Let $E_\CC$ be the complexification of $E$, and let $\Hyper_\CC$ be the set of complexifications of the real hyperplanes from $\Hyper$.  Thus $\Hyper_\CC$ is a locally finite complex hyperplane arrangement in the complex affine space $E_\CC$.

Let $\Cat{Perv}(E_\CC, \Hyper_\CC)$ be the abelian category of perverse sheaves on $E_\CC$, smooth with respect to the stratification induced by $\Hyper_\CC$.  We consider sheaves with coefficients in finite-dimensional complex vector spaces, for convenience.  The main theorem of Kapranov and Schechtman \cite[Theorem 8.1, \S 9.B]{KapSch} is the following.
\begin{thm}
\label{kapsch}
The category $\Cat{Perv}(E_\CC, \Hyper_\CC)$ is equivalent to the category of monotonic, transitive, invertible bisheaves of finite-dimensional $\CC$-vector spaces on $\Facet$.  
\end{thm}

Note that the transitive and invertible properties depend not just on the poset $\Facet$, but on a bit more geometric structure (a notion of collinearity, at least).  In \cite{KapSch}, it is noted that the notions of transitivity and invertibility can be defined if one considers not just the poset $\Facet$ but also the oriented matroid structure arising from $\Hyper$.  

From this theorem, we believe it makes sense to {\em define} the category $\Cat{Perv}(E,\Hyper)$ (without complexification) as the full subcategory of bisheaves $\Cat{BiSh}(\Facet, \Omega)$ whose objects are monotonic, transitive, and invertible.  Here $\Omega$ denotes the category of finite-dimensional complex vector spaces.  With this notation, the main theorem of \cite{KapSch} is the construction of an equivalence of abelian categories,
$$\Cat{Perv}(E_\CC, \Hyper_\CC) \xrightarrow{\sim} \Cat{Perv}(E,\Hyper).$$
The category on the left is of considerable interest to those who study the topology of algebraic varieties (and related representation theory).  The category on the right is amenable to study with linear algebra and combinatorics.

If a group $G$ acts on $E$, and stabilizes (i.e., permutes) the set of hyperplanes $\Hyper$, one may consider $G$-equivariant perverse sheaves.  In the real, combinatorial setting, we define $\Cat{Perv}_G(E,\Hyper)$ to be the full subcategory of $\Cat{BiSh}_G(\Facet, \Omega)$ consisting of monotonic, transitive, invertible bisheaves.  In the complex setting, a $G$-equivariant perverse sheaf is a perverse sheaf $\sheaf{S} \in \Cat{Perv}(E_\CC, \Hyper_\CC)$ endowed with isomorphisms $\iota_g \From \sheaf{S} \To g^\ast \sheaf{S}$, where $g^\ast$ denotes the pullback via $g \From E_\CC \To E_\CC$.  This family of isomorphisms is required to satisfy compatibility conditions with the group structure on $G$, of course.  Note here that we consider $G$ as a discrete group, without any structure as a variety.  Since the $G$-action respects complexification, the equivalence of Kapranov and Schechtman gives an equivalence of abelian categories,
$$\Cat{Perv}_G(E_\CC, \Hyper_\CC) \xrightarrow{\sim} \Cat{Perv}_G(E,\Hyper).$$

\section{$W$-equivariant perverse sheaves on the Coxeter arrangement}

Let us return to the case of a {\em finite} Coxeter system $(W,S)$, with $E$ the (real) reflection representation of $W$ and $\Hyper$ the set of reflection hyperplanes, and $C$ the chamber associated to the set $S$.  Let $\Facet$ be the resulting facet decomposition of $E$.   Recall that $\Lambda$ denotes the set of all subsets of $S$, which is in bijection with the set of facets in $\bar C$.  Here we give a concrete description of the category of $W$-equivariant perverse sheaves on $E_\CC$, smooth with respect to the hyperplane arrangement $\Hyper_\CC$. 

 The case of an {\em infinite affine} Coxeter system should not be much more difficult, but we omit it in order to keep notation a bit simpler.

The theorem of Kapranov and Schechtman gives an equivalence,
$$\Cat{Perv}_W(E_\CC, \Hyper_\CC) \xrightarrow{\sim} \Cat{Perv}_W(E,\Hyper).$$
The category on the right is the full subcategory of $\Cat{BiSh}_W(\Facet, \Omega)$, consisting of those $W$-equivariant bisheaves which are monotonic, transitive, and invertible.  The coefficient category $\Omega$ is the category of finite-dimensional complex vector spaces.

\subsection{The condition of invertibility}

The following result translates the condition of invertibility into a condition on modules over $\AA_W^{\mon}$.  We consider invertibility before transitivity, since the corresponding condition on modules is simpler.  

Recall that $\AA_W^{\mon}$ is the $\CC$-algebra $\CC \langle e_I : I \in \Lambda,\ s : s \in S \rangle$ modulo the relations
\begin{enumerate}
\item
If $I \supset J$ then $e_I e_J = e_J$.
\item
If $s \in I$ then $s e_I = e_I s$.
\item
$s^2 = 1$ for all $s \in S$.
\item
$(s s')^{m_{s,s'}} = 1$ for all $s,s' \in S$.
\end{enumerate}
If $(\sheaf{V}, \gamma, \delta, \eta)$ is a monotonic, $W$-equivariant bisheaf on $\Facet$, then we associated to $\sheaf{V}$ the vector space $V = \sheaf{V}_S$, and $e_I$ acts on $V$ by the composition $\delta_I \circ \gamma_I$, an idempotent endomorphism of $V$.  Each $s \in S$ acts on $V$ by putting $s(v) = \eta(s) v$.

Relations (3) and (4) are those in the Coxeter group $W$.  Thus there is a canonical homomorphism $\CC[W] \To \AA_W^{\mon}$ and we view every element $w \in W$ as an element of $\AA_W^{\mon}$ accordingly.

Proposition \ref{monequiv} demonstrates that this gives an equivalence of categories, from $\Cat{BiSh}_W^{\mon}(\Facet, \Omega)$ to $\Cat{Mod}^{\fd}(\AA_W^{\mon})$.  (We restrict to finite-dimensional modules here.)  To incorporate invertibility, we require a combinatorial lemma.

\begin{lemma}
If $F_1 \mid_{F_0} F_2$ are opposite faces, with $F_1 \subset \bar C$, then there exist $I,J,K \in \Lambda$ and $w \in W_K$ satisfying the following conditions.
\begin{enumerate}
\item
$\# I = \# K - 1 = \# J$ and  $I \cup J \subset K$.
\item
$w \in W_{K}$ satisfies $w J w^{-1} = I$.
\item
$F_1 = F_I$ and $F_0 = F_K$ and $F_2 = w(F_J)$.
\end{enumerate}
\end{lemma}
\begin{proof}
Since $F_1 \subset \bar C$, we have $F_1 = F_I$ for a unique $I \in \Lambda$.  The opposition condition, $F_1 \mid_{F_0} F_2$ states that $\dim(F_2) = \dim(F_1)$, $F_2$ and $F_1$ span the same flat, and $\bar F_2$ contains $F_0$.  The condition $\dim(F_2) = \dim(F_1)$ holds if and only if $F_2 = w(F_J)$ for some $J \in \Lambda$ with $\# J = \# I$.  The conditions $\dim(F_1) = \dim(F_0) + 1$ and $F_0 \subset \bar F_1$ imply that $F_0 = F_K$ for some $K \in \Lambda$ satisfying $\# K = \# I + 1$ and $I \subset K$.  The condition that $F_1$ and $F_2$ span the same flat implies that, for all $t \in I$, $t(F_2) = F_2$.  Equivalently, for all $t \in I$, $w^{-1} t w (F_J) = F_J$.  Equivalently, $W_I = w W_J w^{-1}$.

The condition that $\bar F_2$ contains $F_0$ is equivalent to the condition that $\bar F_J$ contains $w^{-1}(F_0)$.  But this occurs if and only if $\bar F_J \supset F_0$ and $w^{-1}(F_0) = F_0$, since the closed chamber $\bar C$ is a fundamental domain for $W$ and $\bar F_J \subset \bar C$.  Thus $\bar F_2 \supset F_0$ if and only if $w \in W_{K} \text{ and } J \subset K$.  

Thus the opposition $F_1 \mid_{F_0} F_2$ gives subsets $I,J,K \in \Lambda$ such that
$$\# I = \# K - 1 = \# J \text{ and } I \cup J \subset K$$
and there exists $w \in W_K$ satisfying $w W_J w^{-1} = W_I$.  The correspondence sends $(I,J,K,w)$ to $F_1 = F_I$ and $F_0 = F_K$ and $F_2 = w (F_J)$, so only the double-coset of $w$ in $W_I \backslash W_K / W_J$ is relevant.  Finally, we note that if $W_I = w W_J w^{-1}$, then there exists $w'$ in the same double-coset as $w$ such that $I = w J w^{-1}$ (see \cite[Lemma 2]{Sol}).  Replacing $w$ by $w'$ gives the result.
\end{proof}

\begin{figure}
\begin{center}
\begin{tikzpicture}
\fill[black!20] (0,0) -- (2,0) arc (0:60:2) -- cycle;
\draw (30:1) node {$\bar C$};
\foreach \a in {0,60,120,180,240,300}
{ \draw (0,0) -- (\a:2); }
\draw[line width = 2pt] (0.2,0) to node[below] {$F_1 = F_{\{t\}}$} (1.9,0);
\draw[line width = 2pt] (60:0.2) to node[left] {$F_{\{s\}}$} (60:1.9);
\fill (0,0) circle (0.1);
\draw (210:0.6) node[rotate=30] {$F_{\{ s, t \}}$};
\draw[line width = 2pt] (-0.2,0) to node[above] {$F_2$} (-1.9,0);

\draw[<->] (2.2, -0.2) arc (-90:90:0.2);
\draw (2.4,0) node[right] {$t$};

\begin{scope}[rotate=60]
\draw[<->] (2.2, -0.2) arc (-90:90:0.2);
\draw (2.4,0) node[right, rotate=60] {$s$};
\end{scope}

\end{tikzpicture}
\end{center}
\caption{Pictured are opposite facets $F_1 \mid_{F_{\{s,t\}}} F_2$.  If $w = st$, then we find that $w(F_{\{s\} }) = F_2$.  The corresponding opposition in $W$ is $\{ t \} \mid_w \{ s \}$.}
\label{oppofig}
\end{figure}
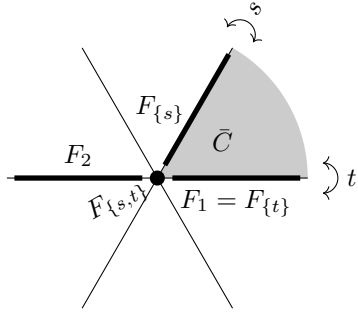

Motivated by this lemma, we write $I \mid_w J$, and say that $I$ and $J$ are {\em opposite via $w$} if $I,J \in \Lambda$, $w \in W_K$ for some $K \in \Lambda$ containing $I \cup J$, and $w J w^{-1} = I$, and $\# I = \# K - 1 = \# J$.  As an important special case, note that $\emptyset \mid_s \emptyset$ for all $s \in S$; this corresponds to the opposition between the Weyl chamber $C$ and its image under any $s \in S$.

\begin{proposition}
\label{invmod}
Let $(\sheaf{V}, \gamma, \delta, \eta)$ be a monotonic, $W$-equivariant bisheaf on $\Facet$, and $V$ the associated $\AA_W^{\mon}$-module.  Then $\sheaf{V}$ is invertible if and only if, for all $I \mid_w J$, $e_J w e_I$ is an isomorphism from $e_I V$ to $e_J V$.
\end{proposition}
\begin{proof}
$\sheaf{V}$ is invertible if and only if, for all opposite faces $F_1 \mid_{F_0} F_2$, the linear map 
$$\gamma_{F_0, F_2} \circ \delta_{F_1, F_0} \From \sheaf{V}_{F_1} \To \sheaf{V}_{F_2}$$
is invertible.  By $W$-equivariance, it suffices to look only at such opposite faces $F_1 \mid_{F_0} F_2$ for which $F_1$ is contained in the closed chamber $\bar C$.  In this case, the lemma gives $I \mid_w J$ such that $F_1 = F_I$,  $F_2 = w(F_J)$.

In this setting, $\delta_{F_1, F_0}$ corresponds to the inclusion of $e_I V$ into $V$.  The vector space $\sheaf{V}_{F_2}$ is identified with $\sheaf{V}_{F_J} = e_J V$ by the isomorphism $\eta(w)$.  The functor from $W$-equivariant bisheaves on $\Facet$ to $W$-equivariant bisheaves on $\Lambda$ sends the linear map $\sheaf{V}_{F_0} \xrightarrow{\gamma} \sheaf{V}_{F_2}$ to the linear map
$$(\sheaf{V}_{F_0} \xrightarrow{\gamma} \sheaf{V}_{F_J}) \circ (\sheaf{V}_{F_0} \xrightarrow{\eta(w)} \sheaf{V}_{F_0}).$$
Thus $\gamma_{F_0, F_2}$ corresponds to the linear map $e_J w \From V \To V_J$.

We find that $\sheaf{V}$ is invertible if and only if $e_J w e_I$ is an invertible linear map from $e_I V$ to $e_J V$.
\end{proof}

The invertibility condition of the previous proposition is equivalent to the invertibility of the map $e_I w^{-1} e_J w e_I \From$ $e_I V \To e_I V$.  Since $e_I$ is idempotent, we may decompose $V = e_I V \oplus (1 - e_I) V$.  Hence the invertibility condition is equivalent to the statement:
$$e_I w^{-1} e_J w e_I + (1- e_I) \text{ is invertible in } \End(V), \text{ for all } I \mid_w J.$$

\subsection{Transitive bisheaves}

The condition of transitivity is more complicated to translate into a condition on modules.  This complication is already recognized, and connected to word-lengths in $W$, by Kapranov and Schechtman in \cite[\S 0.5]{KapSch}.  Part of the challenge is that the transitivity condition is {\em non-local} -- collinear triples can be somewhat far apart, and so $W$-equivariance can only be used to move one facet of a triple $(F_1, F_2, F_3)$ into the closed Weyl chamber.  But from there, collinearity has a well-known relation to minimal galleries in the Coxeter arrangement; see, for example \cite[\S VI.6, Remark]{Brown}.  Minimal galleries then relate to reduced words in the Weyl group, e.g., by \cite[Proposition 2.12, Corollary 2.23]{Tits}.  We apply this combinatorial geometry below.

Define $\AA_W^{\trans}$ as the quotient of $\AA_W^{\mon}$ by the set of relations below.  
\begin{enumerate}
\item[(5)]
Let $S = I \sqcup J$ be a partition of $S$, and let $A$ and $B$ be subsets of $I$.  Write $w_A, w_B$ for the longest elements in the Coxeter groups $(W_A, A)$ and $(W_B, B)$, respectively.  Then if $w, w_1, w_2 \in W_I$, and $w = w_2 w_1$, and 
$$\ell(w w_B w_A) = \ell(w w_B w^{-1}) + \ell(w_2) + \ell(w_1) + \ell(w_A),$$
the relation (associated to the data $(I,J,A,B,w_1,w_2)$) is
$$e_{A \cup J} \cdot w_1 \cdot e_J \cdot w_2 \cdot e_{B \cup J} = e_{A \cup J} \cdot w \cdot e_{B \cup J}.$$
\end{enumerate}

Before proceeding, we note two consequences of (5).  First, suppose $w = w_1 = w_2 = 1$, and $A$ and $B$ are disjoint subsets of $I$.  Note that $\ell(w_A w_B) = \ell(w_A) + \ell(w_B)$ in this case (it follows from \cite[Proposition 2.4.4]{BB}, for example), so relation (5) implies
$$e_{A \cup J} e_J e_{B \cup J} = e_{A \cup J} e_{B \cup J}.$$
But the left side is $e_J$ by (1).  Note that $J = (A \cup J) \cap (B \cup J)$ since $A$ and $B$ are disjoint.  Hence (1) and (5) imply a strengthening of (1), denoted (1') below.
\begin{enumerate}
\item[(1')]
If $I \cap J = K$ then $e_I e_J = e_K$.
\end{enumerate}

Instead of taking $w$ trivial and $A,B$ arbitrary, we may consider the opposite case where $A,B = \emptyset$ and $w = w_2 w_1 \in W_I$.  In this case, relation (5) implies the following (with $J = S - I$ as before),
\begin{enumerate}
\item[($5_I$)]
If $w = w_2 w_1 \in W_I$ and $\ell(w) = \ell(w_2) + \ell(w_1)$, then $e_J \cdot w_2 \cdot e_J \cdot w_1 \cdot e_J = e_J \cdot w \cdot e_J$.
\end{enumerate}
 
An important special case of ($5_I$) is the case $I = \emptyset$, described below.  
\begin{proposition}
\label{braidprop}
For $s \in S$, write $s_0 = e_\emptyset s e_\emptyset \in \AA_W^{\trans}$.  Then $\{ s_0 : s \in S \}$ satisfies the braid relations.  Namely, if $s,t \in S$ and $s \neq t$, then
$$\underbrace{s_0 t_0 s_0 \cdots}_{m_{s,t} \text{ factors}} = \underbrace{t_0 s_0 t_0 \cdots}_{m_{s,t} \text{ factors}} .$$
\end{proposition}
\begin{proof}
This follows from the relation ($5_\emptyset$) above, since the elements $\{ s : s \in S \}$ satisfy the braid relations.
\end{proof}

Every $\AA_W^{\trans}$-module can be viewed as an $\AA_W^{\mon}$-module by pulling back, and this identifies $\Cat{Mod}^{\fd}(\AA_W^{\trans})$ as a full subcategory of $\Cat{Mod}^{\fd}(\AA_W^{\mon})$.  In parallel, write $\Cat{BiSh}_W^{\trans}(\Facet, \Omega)$ for the full subcategory of $\Cat{BiSh}_W^{\mon}(\Facet, \Omega)$ consisting of monotonic {\em and } transitive bisheaves.  (We never consider transitivity without monotonicity.)  Our main structural theorem here relates transitivity of bisheaves to relation (5) discussed above.

\begin{thm}
\label{transthm}
The equivalence of categories $\Cat{BiSh}_W^{\mon}(\Facet, \Omega) \xrightarrow{\sim} \Cat{Mod}^{\fd}(\AA_W^{\mon})$ restricts to an equivalence $\Cat{BiSh}_W^{\trans}(\Facet, \Omega) \xrightarrow{\sim} \Cat{Mod}^{\fd}(\AA_W^{\trans})$
\end{thm}
\begin{proof}
Let $(\sheaf{V}, \gamma, \delta, \eta)$ be a monotonic, $W$-equivariant bisheaf on $\Facet$, and $V$ the associated $\AA_W^{\mon}$-module.  We must show that $\sheaf{V}$ is transitive if and only if the $\AA_W^{\mon}$-module $V$ respects the relations of (5).

By Lemma \ref{transgen}, $\sheaf{V}$ is transitive if and only if for every collinear triple of facets $(F_1, F_2, F_3)$, with $F_2$ in general position, we have
$$\phi_{F_2, F_3} \circ \phi_{F_1, F_2} = \phi_{F_1, F_3}.$$
Parameterize a line $\lambda$ through $F_1, F_2, F_3$, so that $\lambda(1) \in F_1$, $\lambda(2) \in F_2$, and $\lambda(3) \in F_3$.  Define $F_1^\pm$ to be the facet containing $\lambda(1 \pm \epsilon)$ and $F_3^\pm$ the facet containing $\lambda(3 \pm \epsilon)$, for sufficiently small positive $\epsilon$.  See Figure~\ref{TransFig}.
\begin{figure}[htbp!]
\begin{tikzpicture}
	[dotty/.style={
         circle,
         fill=black!80,
         thick,
         inner sep=2pt,
         minimum size=0.15cm}]
\clip (-5,-3) rectangle (5,3);
\fill[black!20] (1,3) -- (0,0) -- (6,3) -- cycle;
\filldraw (0,0) circle (0.05);
\draw (-6,-3) -- (6,3);
\draw (-5, 2) -- (5,-2);
\draw (-1, -3) -- (1,3);
\draw (-1.5, 3) -- (1.5,-3);

\node[dotty] (F1m) [label=below:{$p_1^-$}] at (2.5,1) {};
\node[dotty] (F1) [label=below:{$p_1$}] at (2,1) {};
\node[dotty] (F1p) [label=above:{$p_1^+$}] at (1.5,1) {};
\node[dotty] (F2) [label=above:{$p_2$}] at (-1.75,1) {};
\node[dotty] (F3m) [label=below:{$p_3^-$}] at (-3.5,1) {};
\node[dotty] (F3) [label=below:{$p_3$}] at (-4,1) {};
\node[dotty] (F3p) [label=below:{$p_3^+$}] at (-4.5,1) {};

\draw (2,2.5) node {$F_1^+ \subset \bar C$};

\draw[dotted, <->] (-4.9, 1) -- (4,1) node[right] {$\lambda$};
\end{tikzpicture}
\caption{A 2-plane containing $\lambda$ and the origin.  Intersections of root hyperplanes and the 2-plane are drawn here; note that multiple root hyperplanes may intersect in a single line in the picture above.  Here $p_j = \lambda(j)$ for $j = 1,2,3$, and $p_1^\pm = \lambda(1 \pm \epsilon)$ and $p_3^\pm = \lambda(3 \pm \epsilon)$.  The entire 2-plane is contained in the space~$E_I$.}
\label{TransFig}
\end{figure}
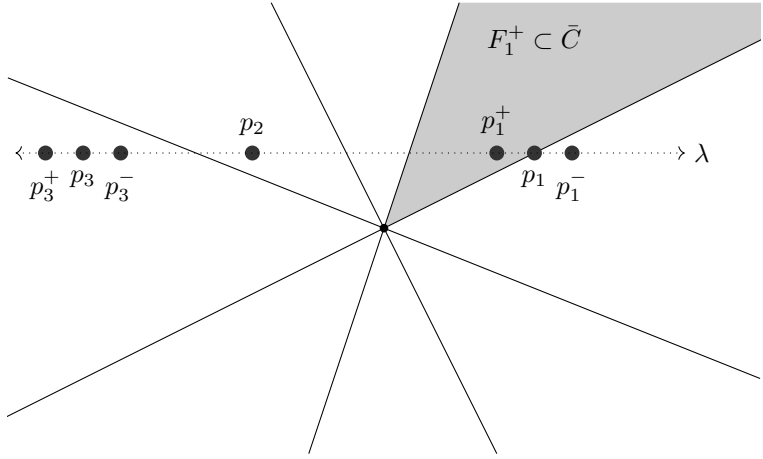

By $W$-equivariance, transitivity is equivalent to transitivity for such triples $(F_1, F_2, F_3)$ satisfying $F_1^+ \in \bar C$.  Let $J \in \Lambda$ be the subset for which $F_1^+ = F_J$.  Let $I = S - J$.  Since $F_1^+$ is in general position along the line through $F_1, F_2, F_3$, we find that
$$F_1, F_2, F_3 \subset E_I = \bigcap_{s \in J} H_s.$$
In other words, the roots in $J$ vanish identically on $F_1, F_2, F_3, F_1^\pm, F_3^\pm$.  Note that 
$$\dim(F_1^+) = \dim(E_I) = \dim(E) - \# J = \# I.$$

Now we work within the induced hyperplane arrangement $(E_I, \Hyper_I)$, where $\Hyper_I = \{ H \cap E_I : H \in \Hyper \}$.  We find that $F_1^+$ and $F_3^-$ are chambers in $E_I$, and so there is a unique element $w \in W_I$ satisfying 
$$w F_1^+ = F_3^-.$$
As $F_1 \leq F_1^+$, and $F_3 \leq F_3^-$, there exist unique subsets $A, B \subset I$ such that
$$F_1 = F_{A \cup J} \text{ and } F_3 = w \left( F_{B \cup J} \right).$$
Moreover, $F_1^-  = w_A F_1^+$ and $F_3^+ = w w_B w^{-1} F_3^-$.  Indeed, crossing the wall from $F_1^+$ to $F_1^-$ negates precisely the roots from $A$.  Looking at the ``link'' of the Coxeter complex $(E_I, \Hyper_I)$ at the facet $F_1$ yields the Coxeter complex associated to $(W_A, A)$, and thus the longest element $w_A$ of $W_A$ relates the opposite facets $F_1^+$ and $F_1^-$.  A similar analysis applies to $F_3^\pm$.

Let $w_1, w_2$ be the unique elements of $W_I$ satisfying
$$w_1 F_1^+ = F_2, \quad w_2 F_2 = F_3^-.$$
Thus $w_2 w_1 = w$.

The line $\lambda$, passing through $F_1^-, F_1^+, F_2, F_3^-, F_3^+$ (chambers in $E_I$), corresponds to a decomposition,
$$\ell( w w_B w^{-1} \cdot w_2 \cdot w_1 \cdot w_A ) = \ell(w w_B w^{-1}) + \ell(w_2) + \ell(w_1) + \ell(w_A).$$

This construction gives a bijection between two sets.
\begin{itemize}
\item
The set of collinear triples of distinct facets $(F_1, F_2, F_3)$, with $F_1^+ \subset \bar C$ and $F_2$ in general position.
\item
The set of data $(I,J,A,B, w_1, w_2)$ in which $S = I \sqcup J$, $A,B \subset I$, $w_1, w_2 \in W_I$, $w = w_2 w_1$, and
$$\ell(w w_B w_A ) = \ell(w w_B w^{-1}) + \ell(w_2) + \ell(w_1) + \ell(w_A).$$
\end{itemize}
The map in one direction has been unambiguously defined, since the triple $(F_1, F_2, F_3)$ determines $F_1^\pm$ and $F_2^\pm$, indepently of the choice of line $\lambda$.

In the other direction, given $(I,J,A,B, w_1, w_2)$, we may choose reduced words for $w w_B w^{-1}, w_2, w_1, w_A$, and concatenate them to give a reduced word for $w w_B w_A$.  Define $F_1^+ = F_J$, $F_1 = F_{A \cup J}$.  Define $F_2 = w_1 F_1^+$ and $F_3^- = w_2 F_2$.  Define $F_3 = w_2 w_1 F_{B \cup J}$.  By construction, the gallery through $F_1, F_1^+, F_2, F_3^-, F_3$ is minimal, and so we find a collinear triple of distinct facets $(F_1, F_2, F_3)$ with $F_1^+ \subset \bar C$ and $F_2$ in general position.  The reader may check that this gives a bijection.

We have now proven that $\sheaf{V}$ is transitive if and only if, for all data $(I,J,A,B, w_1, w_2)$ as above, there is an equality of maps,
\begin{equation}
\label{transphi}
\phi_{w_1 F_J, w_2 w_1 F_{B \cup J}} \circ \phi_{F_{A \cup J}, w_1 F_J} = \phi_{F_{A \cup J}, w F_{B \cup J}}.
\end{equation}

We translate this into conditions on the $\AA_W^{\mon}$-module $V = \sheaf{V}_{F_S}$ in a few steps.  First, note that we have $\sheaf{V}_{F_{A \cup J}} = e_{A \cup J} V$.  Thus the right side of \eqref{transphi}, applied to a general element $e_{A \cup J} v \in \sheaf{V}_{F_{A \cup J}}$, becomes
$$\phi_{F_{A \cup J}, w F_{B \cup J}} ( e_{A \cup J} v) = \gamma_{\circ, w F_{B \cup J}} (e_{A \cup J} v) = e_{B \cup J} w \cdot e_{A \cup J} v.$$

For the left side of \eqref{transphi}, we compute
\begin{align*}
\phi_{w_1 F_J, w_2 w_1 F_{B \cup J}} \circ \phi_{F_{A \cup J}, w_1 F_J} ( e_{A \cup J} v) 
&= \phi_{w_1 F_J, w_2 w_1 F_{B \cup J}} \left(  e_J w_1 \cdot e_{A \cup J} v \right), \\
&= \gamma_{\circ, w_2 w_1 F_{B \cup J}} \delta_{w_1 F_J, \circ}  \left(  e_J w_1 \cdot e_{A \cup J} v \right), \\
&= \gamma_{\circ, w_2 w_1 F_{B \cup J}}  \left( w_1^{-1} \cdot e_J w_1 \cdot e_{A \cup J} v \right), \\
&= e_{B \cup J} w_2 w_1 \cdot w_1^{-1} \cdot  e_J w_1 \cdot e_{A \cup J} v, \\
&= e_{B \cup J} \cdot w_2 \cdot e_J \cdot w_1 \cdot e_{A \cup J} v.
\end{align*}

Therefore, $\sheaf{V}$ is a transitive bisheaf if and only if the $\AA_W^{\mon}$-module $V$ satisfies relation (5) for all appropriate $(I,J,A,B,w_1,w_2)$.
\end{proof}

\subsection{Perverse sheaves}
Together with Kapranov and Schechtman's theorem (Theorem \ref{kapsch}), Theorem \ref{transthm} and Proposition \ref{invmod} imply the following.
\begin{thm}
\label{mainth1}
The category $\Cat{Perv}_W(E_\CC, \Hyper_\CC)$ is equivalent to the category of finite-dimensional modules over the $\CC$-algebra $\AA_W$ generated by $\{ e_I : I \in \Lambda \}$ and $\{ s : s \in S \}$, subject to the relations,
\begin{enumerate}
\item[\rm (1')]
$e_I e_J = e_{I \cap J}$ for all $I,J \in \Lambda$;
\item[\rm (2)]
If $s \in I$ then $s e_I = e_I s$;
\item[\rm (3)]
$s^2 = 1$ for all $s \in S$;
\item[\rm (4)]
$(s s')^{m_{s,s'}} = 1$ for all $s,s' \in S$ with $m_{s,s'} < \infty$;
\item[\rm (5)]
Let $S = I \sqcup J$ be a partition of $S$, and let $A$ and $B$ be subsets of $I$.  Write $w_A, w_B$ for the longest elements in the Coxeter groups $(W_A, A)$ and $(W_B, B)$, respectively.  Then if $w, w_1, w_2 \in W_I$, and $w = w_2 w_1$, and 
$$\ell(w w_B w_A) = \ell(w w_B w^{-1}) + \ell(w_2) + \ell(w_1) + \ell(w_A),$$
the relation is
$$e_{A \cup J} \cdot w_1 \cdot e_J \cdot w_2 \cdot e_{B \cup J} = e_{A \cup J} \cdot w \cdot e_{B \cup J},$$
\end{enumerate}
then localized at the multiplicative subset generated by
$$ \{ e_I w^{-1} e_J w e_I + (1- e_I) \text{ for all } I \mid_w J. \} $$
\end{thm}

At least a few words are in order about the algebra described above.  We are localizing the finitely-generated $\CC$-algebra $\AA_W^{\trans}$ generated by $\{ e_I : I \in \Lambda, s : s \in S \}$, modulo relations (1-5) above.  Let $M$ be the multiplicative subset of $\AA_W^{\trans}$ generated by $\{ e_I w^{-1} e_J w e_I + (1- e_I) \text{ for all } I \mid_w J. \}$.  The localization which we call $\AA_W$ is the universal $\CC$-algebra, endowed with a homomorphism $\iota \From \AA_W^{\trans} \To \AA_W$, such that
\begin{itemize}
\item
If $m \in M$, then $\iota(m)$ is invertible.
\item
If $\alpha \From \AA_W^{\trans} \To R$ is another $\CC$-algebra homomorphism such that $\alpha(m)$ is invertible for all $m \in M$, then there exists a unique homomorphism $\AA_W \To R$ making the obvious diagram commute.
\end{itemize}

Such a universal $M$-inverting ring $\AA_W$ exists, for purely abstract reasons (the Cohn localization).  In general, noncommutative localization can behave badly (e.g., it could yield the zero ring).  In this article, we do not attempt to prove that this localization behaves well, e.g., we do not attempt to prove the Ore conditions.  But the localization does not behave {\em too} badly since $\AA_W$ contains some interesting nontrivial subrings.

By construction (see relations (3), (4) above), there is a canonical injective ring homomorphism 
$$\CC[W] \Into \AA_W,$$ 
sending each generator $s$ in the group ring to the corresponding element of $\AA_W$.  Injectivity follows from the following observation:  the composition
$$\CC[W] \To \AA_W \Onto \AA_W / (e_I : I \neq S)$$ 
is an isomorphism.

There is also an injective ring homomorphism 
$$\frac{\CC[e_I : I \in \Lambda]}{e_I e_J = e_{I \cap J}} \Into \AA_W.$$
The algebra on the left is commutative, generated by the idempotents $e_{S - \{s \} }$ for all $s \in S$.  In this way, every $\AA_W$-module $V$ can be viewed as a $W$-module, graded by the power set of $S$.  The $W$-module structure and gradation partially commute:  the graded piece $V_I = e_I V$ becomes a module for the parabolic subgroup $W_I$.

We have assumed that $W$ is a finite Coxeter group in this theorem.  In the affine setting, we expect a similar result, with the power set $\Lambda$ replaced by the set of ``finitary subsets'' of $S$ -- subsets of $S$ which generate a finite subgroup of $W$.  This may complicate relation (5) in the algebra $\AA_W$.

\subsection{Representations of $W$ and the braid group}
The simplest $W$-equivariant perverse sheaves on $E_\CC$ are those supported at $\{ 0 \}$.  Such a perverse sheaf corresponds to a $\CC[W]$-module $V$.  The corresponding module over $\AA_W$ is the one with the same underlying $\CC[W]$-module $V$ and on which every non-identity idempotent $e_I$ (i.e., those with $I \neq S$) acts as the zero endomorphism.

More interesting are the perverse sheaves arising from local systems on the open stratum.  Let us denote $E_\CC^\circ = \left( E_\CC - \bigcup_{H \in \Hyper} H_\CC \right)$ the open stratum in $E_\CC$.  Write $j \From E_\CC^\circ \Into E_\CC$ for the inclusion.  Choose a point $c \in C$ (the open chamber in $E$) to serve as a base point in $E_\CC^\circ$.  

If $\sheaf{V}$ is a perverse sheaf on $E_\CC$, smooth with respect to $\Hyper_\CC$, then the pullback $j^\ast \sheaf{V}$ is a local system (up to shifting by dimension) on $E_\CC^\circ$.  Write $\sheaf{L} = j^\ast \sheaf{V}[d]$ for this local system.  When $\sheaf{V}$ is $W$-equivariant, $\sheaf{L}$ is a $W$-equivariant local system.  In the other direction, if $\sheaf{L}$ is a $W$-equivariant local system on $E_\CC^\circ$, then the intermediate extension functor defines a $W$-equivariant perverse sheaf $j_{!\ast} \sheaf{L}[d]$ on $E_\CC$.

Since $W$ acts freely on $E_\CC^\circ$, the category of $W$-equivariant local systems on $E_\CC^\circ$ is equivalent to the category of local systems on the quotient $X_W = E_\CC^\circ / W$.  This, in turn, is equivalent to the category of finite-dimensional complex representations of the {\em braid group}
$$\Gamma_W = \pi_1(X_W, x),$$
where $x$ is the image of $c \in E_\CC^\circ$ in $X_W$.  The topology of $X_W$ and structure of the braid group $\Gamma_W$ is by now a classical story.  The space $X_W$ is the ``space of regular orbits'' for the (complex) reflection group $W$, and its fundamental group $\Gamma_W$ was studied by Brieskorn \cite{Bri}.  By construction, there is a surjective homomorphism $\Gamma_W \Onto W$.  Brieskorn \cite[Satz]{Bri} proves that $\Gamma_W$ is generated by elements $\{ \gamma_s : s \in S \}$, with $\gamma_s \in \Gamma_W$ lying over $s \in W$, subject to the braid relations:
$$\underbrace{\gamma_s \gamma_t \gamma_s \cdots}_{m_{s,t} \text{ terms}} = \underbrace{\gamma_t \gamma_s \gamma_t \cdots}_{m_{s,t} \text{ terms}}.$$
In this way, a $W$-equivariant local system on $E_\CC^\circ$ is given by a finite-dimensional representation $M$ of $\Gamma_W$, which is the same as a space $M$ endowed with invertible monodromy operators $\mu_s \in \mathrm{GL}(M)$ for all $s \in S$, satisfying the braid relations above.

\begin{proposition}
Let $\sheaf{V}$ be a $W$-equivariant perverse sheaf on $E_\CC$, smooth with respect to $\Hyper_\CC$.  Write $V$ for the resulting $\AA_W$-module.  Write $(j^\ast \sheaf{V})_W$ for the local system on $X_W = E_\CC^\circ / W$ obtained from the $W$-equivariant local system $j^\ast \sheaf{V}$ on $E_\CC^\circ$.  Then $(j^\ast \sheaf{V})_W$ corresponds to the vector space $M = e_\emptyset V$ with monodromy operators
$$\mu_s = e_\emptyset s e_\emptyset.$$
\end{proposition}
\begin{proof}
Write $\sheaf{V}$ also for the $W$-equivariant bisheaf on $\Facet$ associated to $\sheaf{V}$.  Then, following the comment of \cite[\S0.3 (d)]{KapSch}, the stalk of $j^\ast \sheaf{V}$ at the base point $c \in C$ (an open Weyl chamber) coincides with the space $M \defeq \sheaf{V}_C = e_\emptyset V$.    

The monodromy operators $\mu_s$ are determined by the ``half-monodromy'' operators on the bisheaf $\sheaf{V}$, according to a recipe found in \cite[\S 9.A]{KapSch}.  Namely, if $s \in S$, then we have opposite facets,
$$s F_\emptyset \mid_{F_{ \{s \} } } F_\emptyset.$$
Following loc.~cit., define $E_+ = \sheaf{V}_{F_\emptyset}$, $E_- = \sheaf{V}_{s F_{\emptyset}}$, $E_0 = \sheaf{V}_{F_{ \{ s \} } }$.  Write $\gamma_{\pm} \From E_0 \To E_{\pm}$ and $\delta_{\pm} \From E_{\pm} \To E_0$ for the uppers and downers between $E_0$ and $E_\pm$.  

In terms of $\AA_W$-modules, $E_+ = E_- = e_\emptyset V$, and $E_0 = e_{ \{ s \} } V$.  The half-monodromy, in the fundamental groupoid of $E_\CC^\circ$, is the operator $\gamma_- \delta_+ \From E_+ \To E_-$.  In terms of $\AA_W$-modules, this is the operator $e_\emptyset s$ from $e_\emptyset V \To e_\emptyset V$.  It is more convenient to write this as the operator $e_\emptyset s e_\emptyset$.  

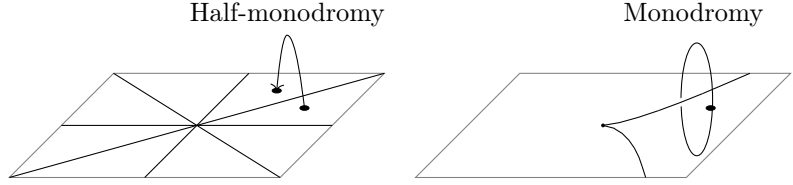
\begin{figure}
\begin{tikzpicture}[scale=0.6]
\draw[gray] (-3,0,-3) -- (3,0,-3) -- (3,0,3) -- (-3,0,3) -- cycle;
\draw (-3,0,0) -- (3,0,0);
\draw (0,0,-3) -- (0,0,3);
\draw (-3,0,-3) -- (3,0,3);
\draw (-3,0,3) -- (3,0,-3);

\coordinate (A) at (2,0,-1);
\coordinate (B) at (1,0,-2);
\filldraw (A) ellipse [x radius = 0.1, y radius = 0.05];
\filldraw (B) ellipse [x radius = 0.1, y radius = 0.05];
\draw[->] (A) parabola bend (2,2,0) (B);
\draw (2,2,0) node[above] {Half-monodromy};

\begin{scope}[xshift = 9cm]
\draw[gray] (-3,0,-3) -- (3,0,-3) -- (3,0,3) -- (-3,0,3) -- cycle;
\draw (0,0,0)
\foreach \x/\y in {0/0, 0.3/0.16, 0.6/0.46, 0.9/0.85, 1.2/1.31, 1.5/1.83, 1.8/2.41, 2.1/3}
{
-- (\x, 0, \y)
}
;
\draw (0,0,0)
\foreach \x/\y in {0/0, 0.3/0.16, 0.6/0.46, 0.9/0.85, 1.2/1.31, 1.5/1.83, 1.8/2.41, 2.1/3}
{
-- (\x, 0, -\y)
}
;
\filldraw (0,0,0) circle (0.03);
\coordinate (A) at (2,0,-1);
\coordinate (B) at (1,0,-2);

\filldraw (A) ellipse [x radius = 0.1, y radius = 0.05];
\draw (1.5,0,-1.5) ellipse [x radius = 0.35, y radius = 1.25];

\draw[line width = 3pt, white] (0.9,0,-0.85)
\foreach \x/\y in {1.2/1.31,1.5/1.83 }
{
-- (\x, 0, -\y)
}
;
\draw (0.9,0,-0.85)
\foreach \x/\y in {1.2/1.31,1.5/1.83 }
{
-- (\x, 0, -\y)
}
;

\draw (2,2,0) node[above] {Monodromy};

\end{scope}
\end{tikzpicture}
\caption{The half-monodromy operator is an element of the fundamental groupoid of $E_\CC^\circ$, while the monodromy operator is an element of the fundamental group of $X_W = E_\CC^\circ / W$.}
\end{figure}

The half-monodromy in the fundamental groupoid of $E_\CC^\circ$ gives the monodromy in the fundamental group of $X_W = E_\CC^\circ / W$.  Thus $\mu_s = e_\emptyset s e_\emptyset$ as claimed. 
\end{proof}

Note, from Theorem \ref{mainth1}, that $e_\emptyset s e_\emptyset$ is invertible as an operator on $e_\emptyset V$, and Proposition \ref{braidprop} implies that the operator $\mu_s = e_\emptyset s e_\emptyset$ satisfies the braid relations in accordance with Brieskorn's theorem.  These results together can be rephrased as the following.

\begin{proposition}
The map $\gamma_s \mapsto e_\emptyset s e_\emptyset$ gives a ring homomorphism from the braid group ring $\CC[\Gamma_W]$ to the subalgebra $e_\emptyset \AA_W e_\emptyset$.
\end{proposition}

The pullback of $W$-equivariant perverse sheaves on $(E_\CC, \Hyper_\CC)$ to local systems on the open stratum corresponds to applying $e_\emptyset$, and then pulling back from $e_\emptyset \AA_W e_\emptyset$-modules to $\CC[\Gamma_W]$-modules.  It would be interesting to study the {\em intermediate extension} functor, in the other direction.

\subsection{Examples}

\subsubsection{The rank one case}

Begin with the case $W = \{ 1, s \}$, acting on $E = \RR$ by $\sigma(x) = -x$.  Then $E_\CC = \CC$, and the unique hyperplane is $H_\CC = \{ 0 \}$.  Write $i \From \{ 0 \} \Into \CC$ and $j \From \CC^\times \Into \CC$ for the inclusions of the closed and open strata, respectively.

The algebra $\AA_W$ is then the $\CC$-algebra generated by two elements $\{ s, e_0 \}$, modulo the relations $s^2 = 1$, $e_0^2 = e_0$, and localized by inverting $e_0 s e_0 + (1 - e_0)$.  The category $\Cat{Perv}_W(E_\CC, \Hyper_\CC)$ of $W$-equivariant perverse sheaves on $\CC$, smooth with respect to the stratification $\{ 0 \} \subset \CC$, is equivalent to the category of finite-dimensional $\AA_W$-modules.  If $\sheaf{V}$ is such a sheaf, we write $V$ for the $\AA_W$-module.

We describe some sheaf-theoretic properties in terms of modules here.  For example, the pullback $j^\ast \sheaf{V}[1]$ is a local system on $\CC^\times$.  The $W$-equivariance of $\sheaf{V}$ implies the $W$-equivariance of the local system, and gives a {\em square root} of the monodromy (from the half-monodromy in the fundamental groupoid).  Explicitly, the local system $j^\ast \sheaf{V}[1]$ corresponds to the vector space $V_0 \defeq e_0 V$, with monodromy operator $(e_0 s e_0)^2$ with its evident square root $\mu = e_0 s e_0$.

A $W$-equivariant perverse sheaf $\sheaf{V}$ is supported on $\{ 0 \}$ if and only if $V = \Ker(e_0)$.  Such sheaves can be constructed from any $\CC[s] / (s^2)$-module $V$ by having $e_0$ act by zero.  

More interesting is the functor of intermediate extension.  Consider a {\em rank one} $W$-equivariant local system $\sheaf{L}$ on $\CC^\times$, with monodromy $\mu^2$ ($W$-equivariance provides a square root of monodromy, which we call $\mu$).  The intermediate extension $j_{!\ast} \sheaf{L}$ is a $W$-equivariant perverse sheaf on $\CC$, extending $\sheaf{L}$, and having no subobject or quotient object supported on $\{ 0 \}$.

If $\mu \neq \pm 1$, then $j_{! \ast} \sheaf{L}$ corresponds to an $\AA_W$-module $V$.  Choosing a suitable basis, i.e., choosing bases of stalks in the two Weyl chambers, we may identify $V = \CC^2$ with module structure given by
$$e_0 = \Matrix{1}{0}{0}{0} \text{ and } s = \Matrix{\mu}{1+\mu}{1 - \mu}{-\mu}.$$

If $\mu = \pm 1$, the intermediate extension corresponds to a smaller $\AA_W$-module $V$.  In this case, we have $V \isom \CC$ with module structure given by
$$e_0 = \Id \text{ and } s = \mu.$$

\subsubsection{Type $A_2$}

Now consider the case $W = \{ 1, s, t, st, ts, sts \}$ acting on $E = \{ (x,y,z) \in \RR^3 : x + y + z = 0 \}$ by the usual reflection representation.  We set $s(x,y,z) = (x,z,y)$ and $t(x,y,z) = (z,y,x)$.  Then $E_\CC \isom \CC^2$, and there are three lines of reflection $H_x, H_y, H_z$ corresponding to the conditions $x = 0$, $y = 0$, $z = 0$.  Note that $s$ fixes $H_x$ and $t$ fixes $H_y$.

In this case, the algebra $\AA_W$ can be described as $\CC \langle s, t, e_x, e_y \rangle$ modulo the following relations (and then localized).
\begin{enumerate}
\item
$e_x^2 = e_x$, $e_y^2 = e_y$, and $e_x e_y = e_y e_x$.  Write $e_0 = e_x e_y$.
\item
$s e_x = e_x s$ and $t e_y = e_y t$.
\item
$s^2 = t^2 = 1$.
\item
$sts = tst$.
\item
Write $s_0 = e_0 s e_0$ and $t_0 = e_0 t e_0$.  Then
$s_0 t_0 s_0 = e_0 sts e_0 =  e_0 tst e_0 = t_0 s_0 t_0$.
\end{enumerate}
These relations fully capture the relations (1-5) of $\AA_W^{\trans}$, which simplify considerably in low-rank cases.

\begin{remark}
The category of modules over this algebra is equivalent to the category of $S_3$-equivariant perverse sheaves on $E_\CC \isom \CC^2$, smooth with respect to the stratification by coordinate hyperplanes.  The quotient scheme $\CC^2 / S_3$ can be identified with $\CC^2$, with stratification by a nodal cubic $(y^2 = x^3)$ and $\{ 0 \}$, by classical invariant theory.  The category of $\AA_W$-modules is equivalent to the category of perverse sheaves on the quotient {\em stack}, $[\CC^2 / S_3]$, which captures the left-over equivariance on the 0- and 1-dimensional strata.
\end{remark}

\section{Perverse sheaves on buildings}

If $\Hyper$ is a hyperplane arrangement in a real affine space $E$, with resulting facet decomposition $\Facet$, recall that $\Cat{Perv}(E, \Hyper)$ denotes the category of monotonic, transitive, invertible bisheaves on $\Facet$.  This is the category which Kapranov and Schechtman prove equivalent to the category $\Cat{Perv}(E_\CC, \Hyper_\CC)$ of perverse sheaves on the complex space $E_\CC$, smooth with respect to the hyperplane stratification given by $\Hyper_\CC$.

As affine buildings are built from real hyperplane arrangements (apartments), we extend the category $\Cat{Perv}(E, \Hyper)$ to give a category of ``perverse sheaves on a building''.  Buildings themselves do not seem to have a natural complexification, glued from complex hyperplane arrangements.  But {\em categories} of perverse sheaves on each apartment can be glued nonetheless.  

\subsection{Perverse sheaves on hyperplane complexes}

Before specializing to buildings, we study perverse sheaves on more general spaces glued from hyperplane arrangements.  We define a \defined{hyperplane complex} to be a structure $(\Build, \{ E_i, \Hyper_i \}_{i \in I})$, where $\Build$ is a topological space endowed with a topological decomposition $\Build = \bigcup_{i \in I} E_i$ into closed subspaces $E_i$ called \defined{apartments}, and each $E_i$ is endowed with the structure of a real affine space together with a locally finite set of affine hyperplanes $\Hyper_i$.  Write $\Facet_i$ for the set of facets for the hyperplane arrangement $\Hyper_i$ in $E_i$.  For this triple to be a hyperplane complex, we require the following axioms. 
\begin{enumerate}
\item
For all $F_i \in \Facet_i$ and $F_j \in \Facet_j$, $F_i \cap F_j$ is empty or $F_i = F_j$ (they are equal as subspaces of $\Build$).  In this way, it makes sense to write $\Facet = \bigcup_{i \in I} \Facet_i$ for the poset of facets of $\Build$.  We write $F_1 \leq F_2$ in $\Build$ if there exists an apartment $E$ such that $F_1 \leq F_2$ in $E$.
\item
For any two facets $F_1, F_2 \in \Facet$, there exists an apartment $E$ such that $F_1, F_2 \subset E$.
\item
Let $E_i$ and $E_j$ be two apartments in $\Build$, and let $F_1, F_2$ be two facets in $\Build$.  If $F_1, F_2 \subset E_i$ and $F_1, F_2 \subset E_j$, then there exists an isomorphism of real affine spaces $\phi \From E_i \To E_j$ such that $\phi(\Hyper_i) = \Hyper_j$ and $\phi$ fixes $F_1$ and $F_2$ pointwise.
\end{enumerate}



Let $\Build$ be a hyperplane complex, and $p_1, p_2 \in \Build$ two points.  Then $p_1, p_2$ belong to unique facets $F_1, F_2$ of $\Build$, and by (2) there exists an apartment $E$ containing both $p_1$ and $p_2$.  Write $\overline{p_1 p_2}$ for the geodesic (line segment) in the affine space $E$ joining $p_1$ and $p_2$.  Axiom (3) implies that the geodesic $\overline{p_1 p_2}$ depends only on $p_1$ and $p_2$, and not on the choice of apartment.

Every hyperplane complex $\Build$ yields a poset $\Facet$ of facets.  Thus we may consider the abelian category of (monotonic) bisheaves on $\Build$, with coefficients in finite-dimensional complex vector spaces as usual.  A monotonic bisheaf $\sheaf{V}$ on $\Build$ is called \defined{transitive} if its restriction to every apartment is transitive.  Similarly, it is called \defined{invertible} if its restriction to every apartment is invertible.  In this way, we define the category $\Cat{Perv}(\Build)$ of \defined{perverse sheaves} on the hyperplane complex $\Build$, as the category of monotonic, transitive, and invertible bisheaves on $\Build$.

\subsection{Bisheaves from depth zero representations}

In \cite{SS}, Schneider and Stuhler begin with a smooth representation $\pi$ of a $p$-adic group $G$, and construct an equivariant sheaf and an equivariant cosheaf on the building.  An irreducible smooth representation $\pi$ can be recovered in the cohomology or homology thereof.  Using their sheaves, Schneider and Stuhler give a beautiful recipe to understand the character of an irreducible representation $\pi$ (on elliptic elements) and shed light on a mysterious duality of Bernstein, Zelevinski, and Aubert.

Here we demonstrate that Schneider and Stuhler's construction gives not just a sheaf and cosheaf on the building, but  when $\pi$ has {\em depth zero}, their construction gives a perverse sheaf on the building.  

We fix a bit of notation for what follows.  Let $k$ be a nonarchimedean local field.  Let $\OO$ be the ring of integers in $k$, $\mm$ the maximal ideal in $\OO$, and $\FF_q$ the residue field $\OO / \mm$ (of cardinality $q$).  Let $\alg{G}$ be a connected reductive group over $k$, and $G = \alg{G}(k)$.  Let $\Build$ be the enlarged (accounting for the center) Bruhat--Tits building of $G$.  If $\alg{S}$ is a maximal $k$-split torus in $\alg{G}$, we write $\Apart(\alg{S})$ for the associated apartment, $\Phi(\alg{S})$ for the set of roots of $\alg{G}$ with respect to $\alg{S}$ (over $k$), and $\Psi(\alg{S})$ for the set of affine roots.  

Each apartment $\Apart$ in the building $\Build$ is (after a base point is chosen) a real vector space, endowed with an affine hyperplane arrangement $\Hyper$ from the vanishing loci of affine roots.  Let $(\Facet, \leq)$ be the poset of facets in the building $\Build$.  The building is an example of a hyperplane complex, in the terminology of the previous section.

For every facet $F \in \Facet$, write $G_F$ for the corresponding parahoric subgroup, following the convention that $G_F = \sch{G}_F(\OO)$ for the Bruhat--Tits group scheme $\sch{G}_F$ over $\OO$ with {\em connected} special fibre $\alg{\bar G}_F$.  Write $G_F^+$ for the pro-unipotent radical of $G_F$.  If $F \leq F'$ then 
$$G_F^+ \subset G_{F'}^+ \subset G_{F'} \subset G_{F}.$$

Now let $(\pi, V)$ be a smooth representation of $G = \alg{G}(k)$.  For every $F \in \Facet$, define $\sheaf{V}_F = V^{G_F^+}$, the space of $G_F^+$-fixed vectors.  For $F_1 \leq F_2$, inclusion gives a map,
$$\delta_{F_2,F_1} \From \sheaf{V}_{F_2} = V^{G_{F_2}^+} \Into V^{G_{F_1}^+} = \sheaf{V}_{F_1}.$$
On the other hand, projection gives a map,
$$\gamma_{F_1,F_2} \From \sheaf{V}_{F_1} = V^{G_{F_1}^+} \Onto V^{G_{F_2}^+} = \sheaf{V}_{F_2}.$$
Explicitly,
$$\gamma_{F_1,F_2}(v) = \frac{1}{\Vol(G_{F_2}^+)} \int_{G_{F_2}^+} \pi(g) v ~ dg.$$
The axioms for a bisheaf are easily checked in this case.  In fact, $(\sheaf{V}, \gamma, \delta)$ is a monotonic bisheaf.  For if $K_1 \subset K_2$ are compact open subgroups of $G$, then inclusion followed by projection -- from $K_2$-fixed vectors to $K_1$-fixed vectors to $K_2$-fixed vectors -- is the identity map.

This can be extended to an exact functor,
$$\Cat{Rep}^{\sm}(G) \To \Cat{BiSh}^{\mon}(\Build).$$
Here $\Cat{Rep}^{\sm}(G)$ denotes the category of smooth representations of $G$ and $G$-intertwining maps, and $\Cat{BiSh}^{\mon}(\Facet)$ denotes the category of monotonic bisheaves of complex vector spaces.  The functor is exact, as the operation of taking fixed vectors under a profinite group is exact.

After Moy and Prasad, \cite{MP1}, an {\em irreducible} smooth representation $(\pi, V)$ has \defined{depth zero} if $V^{G_F^+} \neq 0$ for some $F \in \Facet$; otherwise, we say it has positive depth.  More generally, any smooth representation $(\pi, V)$ is said to have depth zero if all irreducible subquotients of $(\pi, V)$ have depth zero.  The abelian category $\Cat{Rep}^{\sm}(G)$ splits as a direct product $\Cat{Rep}^{\sm}(G) = \Cat{Rep}^{\sm}_0(G) \times \Cat{Rep}^{\sm}_{>0}(G)$, where the first factor denotes the full subcategory of depth zero representations.  An object of the second factor $\Cat{Rep}^{\sm}_{>0}(G)$ is a smooth representation, all of whose irreducible subquotients have positive depth.  See \cite[\S 1.2(c), 6.2]{BKY} for more details on splitting the category.  The {\em depth zero projector} sends every smooth representation to its maximal summand of depth zero.    

The functor $\Cat{Rep}^{\sm}(G) \To \Cat{BiSh}^{\mon}(\Build)$ sends irreducible depth zero representations to nonzero bisheaves.  If an irreducible representation $(\pi, V)$ has positive depth, then it has no $G_F^+$-fixed vectors at any facet $F$, and hence the functor sends $(\pi, V)$ to zero.  Hence the functor $\Cat{Rep}^{\sm}(G) \To \Cat{BiSh}^{\mon}(\Build)$ factors through the depth zero projector (cf.~the recent work of Barbasch, Ciubotaru, and Moy \cite{BCM}).

The bisheaf $(\sheaf{V}, \gamma, \delta)$ associated to a smooth representation $(\pi, V)$ has a natural $G$-equivariant structure.  Indeed, if $g \in G$, then $\pi(g)$ induces an isomorphism,
$$\eta_{g,F} := \pi(g) \From \sheaf{V}_F = V^{G_F^+} \To V^{g G_F^+ g^{-1}} = V^{G_{g(F)}^+} = \sheaf{V}_{g(F)}.$$
The axioms for a $G$-equivariant structure are straightforward to check, and we find that $(\sheaf{V}, \gamma, \delta, \eta)$ is a $G$-equivariant monotonic bisheaf on $\Facet$.  This gives a functor,
$$\Cat{Rep}^{\sm}(G) \To \Cat{BiSh}_G^{\mon}(\Build).$$

This entire construction is the same as the construction of Schneider and Stuhler \cite{SS}, in a slightly different language.  The primary difference is that we consider the sheaf and cosheaf structure simultaneously, and we verify further axioms on the bisheaf below.

\subsection{Invertibility}

The proof of invertibility is a bit easier than the proof of transitivity, so this is where we begin.  On the other hand, the proof given here of invertibiliy relies on deeper results in the representation theory of finite groups of Lie type.  In particular, it is not clear how the condition of invertibility, and thus the construction of perverse sheaves, might be adapted beyond depth zero.

\begin{proposition}
If $(\pi, V)$ is a smooth representation of $G$, then the associated bisheaf $(\sheaf{V}, \gamma, \delta)$ on $\Build$ is invertible.
\end{proposition}
\begin{proof}
The proof is the same as \cite[Proposition 6.1 (2)]{MP2}, and we review the structural details following their treatment.  Let $F_1, F_2$ be facets of the same dimension $r$, lying in a common $r$-dimensional flat in some apartment, on opposite sides of an $(r-1)$-dimensional face $F_0$.  

Let $\alg{\bar G}_0$ be the reductive group over $\FF_q$ which occurs as the reductive quotient of the special fibre for the parahoric at $F_0$.  Then there exist {\em associate} parabolic subgroups $\alg{\bar P}_1 = \alg{\bar M} \alg{\bar N}_1, \alg{\bar P}_2 = \alg{\bar M} \alg{\bar N}_2$ in the reductive group $\alg{\bar G}_0$, with identifications
$$G_{F_1} / G_{F_0}^+ = \alg{\bar P}_1(\FF_q), \quad G_{F_2} / G_{F_0}^+ = \alg{\bar P}_2(\FF_q), \text{ and } G_{F_1} / G_{F_1}^+ = G_{F_2} / G_{F_2}^+ = \alg{\bar M}(\FF_q).$$

Define $V_i = \sheaf{V}_{F_i} = V^{G_{F_i}^+}$, for $i = 0,1,2$.  Then $V_0$ is a representation of $G_{F_0}$, and $V_1, V_2$ are the fixed vectors in $V_0$ for the actions of $G_{F_1}^+, G_{F_2}^+$ respectively.  These representations are trivial on $G_{F_0}^+$ however, so we may instead say that $V_0$ is a representation of $\alg{\bar G}_0(\FF_q)$, and $V_1, V_2$ are the fixed vectors for $\alg{\bar N}_1(\FF_q) = G_{F_1}^+ / G_{F_0}^+$ and $\alg{\bar N}_2(\FF_q) = G_{F_2}^+ / G_{F_0}^+$, respectively.

The map $\gamma_{F_0, F_2} \delta_{F_1, F_0} \From V_1 \To V_2$ can now be identified with the map
$$v \mapsto \frac{1}{\# \alg{\bar N}_2(\FF_q) } \cdot \sum_{n \in \alg{\bar N}_2(\FF_q)} \pi(n) \cdot v.$$
By \cite[Proposition 6.1 (2)]{MP2}, following Howlett and Lehrer \cite{HowLeh}, this map is an isomorphism from $V_1$ to $V_2$.
\end{proof}

\subsection{The transitive property}

The following proposition was proven in collaboration with Jessica Fintzen.
\begin{proposition}
If $(\pi, V)$ is a smooth representation of $G$, then the associated monotonic bisheaf $(\sheaf{V}, \gamma, \delta)$ on $\Build$ is transitive.
\end{proposition}
\begin{proof}
In order to prove that $\sheaf{V}$ is transitive, we must consider a collinear triple $(F_1, F_2, F_3)$ of facets, all contained in a star neighborhood $\Star(F_0)$ in the building $\Build$.  Thus $F_1, F_2, F_3$ correspond to a triple of parabolic subgroups 
$$\alg{\bar P_1} = \alg{\bar M}_1 \alg{\bar N}_1, \quad \alg{\bar P_2} = \alg{\bar M}_2 \alg{\bar N}_2, \quad \alg{\bar P_3} = \alg{\bar M}_3 \alg{\bar N}_3 \subset \alg{\bar G}_0.$$
In what follows, we write $\bar N_j = \alg{\bar N}_j(\FF_q)$.  

Write $V_0 = \sheaf{V}_{F_0} = V^{G_{F_0}^+}$, and similarly for $V_1, V_2, V_3$.   As in the previous section, $V_0$ is naturally a representation of $\bar G_0 = \alg{\bar G}_0(\FF_q)$.  All computations will take place within subspaces of $V_0$, since indeed,
$$V_j = \sheaf{V}_{F_j} = V_0^{\bar N_j^+} \text{ for } j = 1,2,3.$$

For $1 \leq i < j \leq 3$, the map $\phi_{i,j} \defeq \phi_{F_i, F_j}$ is simply averaging over $\bar N_j$.  
$$\phi_{i,j}(v) = \frac{1}{\# \bar N_j } \cdot \sum_{n \in \bar N_j } \pi(n) \cdot v \text{ for all } v \in V_i.$$

Choose an apartment containing $F_0, F_1, F_2, F_3$.  This is possible, since any apartment containing $F_1$ and $F_3$ will contain their closure, hence $F_0$, and any apartment containing $F_1$ and $F_3$ will also contain facets along geodesics joining points of $F_1$ to $F_3$, and hence will contain $F_2$ along the way.

The choice of apartment determines a maximal $k$-split torus $\alg{S} \subset \alg{G}$, and a corresponding torus $\alg{\bar S} \subset \alg{\bar G}_j$, so that the parabolics $\alg{\bar P}_j$ are standard.

For any affine root $\psi$, let $U_\psi$ be the affine root subgroup as defined in \cite[\S 2.4]{MP1}.  In particular, note that $U_\psi \supset U_{2 \psi}$ when the gradient of $\psi$ is multipliable.  Following \cite[\S 2.6]{MP1}, each group $G_{F_i}^+$ is generated by $\{U_\psi : \psi(F_i) > 0 \}$ together with a subset of $Z(S)$, the centralizer of $\alg{S}(k)$.  Write $\bar U_\psi$ for the image of $U_\psi$ modulo $G_{F_0}^+$.

To prove transitivity, $\phi_{1,3} = \phi_{2,3} \circ \phi_{1,2}$, it suffices to prove the following:
\begin{center}
If $\bar U_\psi \subset \bar N_2$ then $\bar U_\psi \subset \bar N_1 \text{ or } \bar U_\psi \subset \bar N_3$.
\end{center}
From the structure theory above, this is equivalent to proving
\begin{center}
If $\psi \vert_{F_2} > 0$ then $\psi \vert_{F_1} > 0$ or $\psi \vert_{F_3} > 0$.
\end{center}
This is clear, since $(F_1, F_2, F_3)$ is a collinear triple and $\psi$ is an affine-linear function.
\end{proof}

\subsection{Perversity and conclusion}

As we have proven monotonicity, invertibility, and transitivity, we have proven the following result.

\begin{thm}
\label{depthzeroPS}
If $(\pi, V)$ is an admissible representation of $G$, then $(\sheaf{V}, \gamma, \delta, \eta)$ is a $G$-equivariant perverse sheaf on the Bruhat--Tits building $\Build$.
\end{thm}

Note that we placed the hypothesis of admissibility, to guarantee that the spaces $\sheaf{V}_F$ are finite-dimensional.  This gives an exact functor, factoring through the depth zero projector, from the category of admissible representations of $G$ to the category of $G$-equivariant perverse sheaves on $\Build$.  Moreover, a depth zero admissible representation $(\pi, V)$ is naturally identified (as a smooth representation of $G$) with the cosheaf homology $H_0(\Build, \sheaf{V})$ according to \cite[Theorem II.3.1]{SS} (with a minor restriction removed by \cite[Theorem 2.4, Lemma 2.6]{MeySol}).  In particular, $(\pi, V)$ can be recovered from the $G$-equivariant perverse sheaf.

The restriction of such a perverse sheaf to an apartment becomes a perverse sheaf on an affine hyperplane arrangement, equivariant for an affine Weyl group.  We have given an explicit description of the category of such perverse sheaves, at least for finite $W$, in terms of modules over an algebra $\AA_W$.  It would be interesting to connect such algebras to the Hecke algebras of a $p$-adic group, e.g., a full depth zero Hecke algebra $C_c^\infty(G_F^+ \backslash G / G_F^+)$ for $F$ a vertex in the building.

The ingredients used to demonstrate perversity (especially the invertibility condition) are not so different from the ingredients used by Moy and Prasad \cite[\S 1, \S 6]{MP2} in order to demonstrate that unrefined minimal K-types are, in fact ``refined''.  Therefore it seems important to determine whether one can functorially associate $G$-equivariant perverse sheaves on $\Build$ to higher-depth representations of $G$.  But a simple approach of taking fixed vectors at higher-depth (e.g., for Moy-Prasad subgroups) cannot by itself lead to a perverse sheaf -- the invertibility condition is not satisfied, nor are K-types refined by such a process.

\bibliographymark{References}

\providecommand{\bysame}{\leavevmode\hbox to3em{\hrulefill}\thinspace}
\providecommand{\arXiv}[1]{\href{https://arxiv.org/abs/#1}{arXiv:#1}}
\providecommand{\MR}{\relax\ifhmode\unskip\space\fi MR }
\providecommand{\MRhref}[2]{%
  \href{http://www.ams.org/mathscinet-getitem?mr=#1}{#2}
}
\providecommand{\href}[2]{#2}

\end{document}